\documentclass[a4paper,12pt]{amsart}
\usepackage[utf8]{inputenc}
\usepackage{color, comment, bm, enumerate}

\usepackage{tikz}
\usepackage{todonotes}

\usepackage[foot]{amsaddr} 

\newtheorem{definition}{Definition}
\newtheorem{proposition}[definition]{Proposition}
\newtheorem{theorem}[definition]{Theorem}

\newtheorem{lemma}[definition]{Lemma}
\newtheorem{remark}[definition]{Remark}
\newtheorem{assumption}{Assumption}

\renewcommand{\d}{\mathrm{d}}
\newcommand{\ep}{\varepsilon}

\newcommand{\wh}{\widehat}
\newcommand{\wt}{\widetilde}
\newcommand{\pr}{\prime}

\newcommand{\R}{\mathbb{R}}

\newcommand{\lt}{\left}
\newcommand{\rt}{\right}
\newcommand{\PP}{\mathbb{P}}
\newcommand{\EE}{\mathbb{E}}

\begin{document}

\title[FBSDEs with distributional coefficients]{Forward-backward SDEs with distributional coefficients}
\author{Elena Issoglio$^{1,*}$}
\address{$^1$Department of Mathematics, University of Leeds, Leeds, LS2 9JT, UK. \texttt{E.Issoglio@leeds.ac.uk}}
\author{Shuai Jing$^{2,\circ}$}
\address{$^2$Department of Management Science, Central University of Finance and Economics, Beijing, 100081, China. \texttt{jing@cufe.edu.cn} }
\address{$^\circ$Partially funded by NSFC No.71772195, 51609270.}
\address{$^*$Corresponding author}



\begin{abstract}
Forward-backward stochastic differential equations (FBSDEs) have attracted significant attention since they were introduced 30 years ago, due to their wide range of applications, from solving non-linear PDEs to pricing American-type options. Here, we consider two new classes of multidimensional FBSDEs with distributional coefficients (elements of a Sobolev space with negative order). We introduce a suitable notion of a solution and show its existence and uniqueness. We establish a link with PDE theory via a nonlinear Feynman-Kac formula. The associated semi-linear second order parabolic PDE is the same for both FBSDEs, also involves distributional coefficients and has not previously been investigated.
\end{abstract}
\maketitle
\section{Introduction} \label{sc: introduction}
In this paper we study  systems of multidimensional forward-backward stochastic differential equations (forward-backward SDEs or FBSDEs for shortness) with generalized coefficients. In particular, we consider a  class of coefficients $b$ which are elements of the space $ L^\infty([0,T], H^{-\beta}_{ q})$ for some $\beta\in (0,1/2)$, where $H^{-\beta}_{ q} $  is a fractional Sobolev space of negative derivation order, hence its elements are distributions (see Section~\ref{sc: preliminaries} for its definition). We consider two different systems of {\em FBSDEs with distributional coefficients}, both decoupled so that the forward equation can be solved first and the solution plugged into the backward equation.

In the first system, the distribution $b$ appears in the driver of the  backward equation as follows
 \begin{equation}\label{eq: forward-backward SDE transformed}
\left\{
\begin{array}{l}	
	 {X}_s^{t,x} = \ x  +\int_t^s \mathrm dW_r,\\
 {Y}_s^{t,x} = \  \Phi({X}_T^{t,x})  - \int_s^T{Z}_r^{t,x} \mathrm d{W}_r + \int_s^T f(r,{X}_r^{t,x},{Y}_r^{t,x}, Z_r^{t,x}) \mathrm dr\\ 	
	\qquad +  \int^T_s {Z}_r^{t,x}b(r,{X}_r^{t,x})\mathrm{d} r,\\
{\forall }  s\in[t,T],
\end{array}\right.
\end{equation}
where $W$ is a  $d$-dimensional Brownian motion, $\Phi$ and $f$ are functions with standard regularity properties which will be specified later, {and the processes $X, Y, Z$ are $d$, $m$ and $m\times d$-dimensional, respectively}.

In the second system,  the distribution appears in the forward equation  as follows
 \begin{equation} \label{eq: forward-backward SDE}
	\left\{
\begin{array}{l}	
 X_s^{t,x} =\  x +  \int_t^s b(r,X_r^{t,x}) \mathrm dr  +\int_t^s \mathrm dW_r,\\ 
	 Y_s^{t,x} = \Phi( X_T^{t,x})  - \int_s^T Z_r^{t,x} \mathrm dW_r + \int_s^T f(r,X_r^{t,x}, Y_r^{t,x}, Z_r^{t,x}) \mathrm dr,\\
	 \forall   s\in[t,T].
	 \end{array}\right.
\end{equation}

The two systems are studied independently. 
  We give a meaning to the integral terms $\int^T_s  {Z}_r^{t,x}b(r,{X}_r^{t,x})\mathrm{d} r $ and $  \int_t^s b(r,X_r^{t,x}) \mathrm dr $ by introducing a suitable notion of solution for the systems \eqref{eq: forward-backward SDE transformed} and \eqref{eq: forward-backward SDE}, and then  investigate their existence and uniqueness. Moreover we look at   the associated PDE and show its link with the FBSDEs (the well known {\em non-linear Feynman-Kac formula}). As one might expect, it turns out that the PDE associated to both systems \eqref{eq: forward-backward SDE transformed}  and  \eqref{eq: forward-backward SDE} is the same, and it is a semi-linear equation of the form
\begin{equation}\label{eq: PDE backward intro}
\left\{\begin{array}{l}
 u_t(t,x) +  L^b u(t,x) + f(t,x,u(t,x), \nabla u(t,x))=0, \\
 u(T,x)=\Phi(x),\\
 \forall  (t,x)\in[0,T]\times \mathbb R^d ,
\end{array}
\right.
\end{equation}
where the operator $L^b u := \frac 12 \Delta u +  \nabla u b$ is defined component by component (see Section \ref{sc: semi-linear PDE}).  This PDE also involves   distributional coefficients, in particular the drift $b$ which is multiplied by $\nabla u$. A thorough investigation of the partial differential equation is carried out.

\vspace{7pt}
\emph{Literature review.} The history of FBSDEs dates back to 1990, when the foundational paper of Pardoux and Peng \cite{PardouxPeng90} appeared.  In 1992 the same authors   established the link between  (decoupled) FBSDEs and quasi-linear PDEs, well-known as the non-linear Feynman-Kac formula \cite{PardouxPeng92}. A year later, Antonelli \cite{antonelli93} studied for the first time fully coupled FBSDEs  in a small time interval. Since then, the theory of BSDEs and of FBSDEs received a lot of attention by the mathematical community and found many  applications in different fields, especially in finance. For  more details on the latter we refer to the paper of El Karoui et al.\ \cite{ElKaroui-et.al.97} and references therein.

The above-mentioned literature and many subsequent papers were concerned with strong solutions, but starting from the early 2000s mathematicians introduced and studied the notion of weak solution  for FBSDEs. Weak solutions  are  analogous to weak solutions for SDEs, and their importance is illustrated by a series of   stochastic differential equations  which admit a  weak solution but for which  no strong solution exists. For example we mention the well-known Tsirel'son's stochastic differential equation introduced in 1975 by Tsirel'son \cite{tsirelson75},  or   the so-called sticky Brownian motion, which was recently studied by Engelbert and Peskir \cite{engelbert-peskir14}.
Antonelli and Ma \cite{AntonelliMa03} first proposed the notion of weak solutions for FBSDEs in 2003. A more  general notion of weak solution was studied later by Buckdahn et al.\ \cite{Buckdahn_et_al04} in 2004, where the equation for the forward component was implicitly given, and its existence without the uniqueness was discussed. Lejay \cite{lejay04} in 2004 studied existence of weak solutions by using the link between FBSDEs and weak and mild solutions of PDEs. Delarue and Guatteri \cite{delarue-guatteri06} in 2006 were the first to establish uniqueness of weak solutions for fully coupled Markovian FBSDEs. In their paper, the coefficients for the backward equation  are Lipschitz, hence  the ``weak'' notion essentially only intervenes in the forward equation. In 2008  Ma et al.\  \cite{ma_et.al08} also studied existence and uniqueness of weak solutions but in a more general framework, and in fact there the ``weak'' character appears both in the forward and in the backward equation.

The literature on FBSDEs is large but to our knowledge there is very little about (forward-)backward equations with generalized functions (Schwartz distributions). In 1997-1998, Erraoui, Ouknine and Sbi \cite{erraoui-ouknine-sbi97, erraoui-ouknine-sbi98} studied (reflected) BSDEs with distribution as terminal condition. By applying the stochastic flow method, Bally and Matoussi \cite{bally_matoussi01} in 2001 studied stochastic PDE with terminal values and coefficients being distributions  using Backward Doubly SDEs. In 2007, Hu and Tessitore \cite{hu_tessitore07} studied mild solutions of elliptic PDEs in Hilbert spaces by proving the regularity properties of a bounded solution of a BSDE with infinite horizon. Recently, Russo and Wurzer  \cite{russo_wurzer15} studied a one-dimensional BSDE indirectly involving distributional coefficients: They consider and solve a   semilinear ODE with  a distributional drift  and study the associated one-dimensional martingale problem. The martingales are then used to construct the solution of  a martingale-driven BSDE with random terminal time. We also cite the recent results of Diehl and Zhang \cite{diehl-zhang17} where the authors deal with BSDEs with Young integrals.

\vspace{7pt}
\emph{Motivation.} 
The importance of classical results on FBSDEs and their link to PDEs through the generalized Feynman-Kac formula is well known. In our case, we relax notably the assumptions on the coefficients of Markovian systems of FBSDEs to allow for generalized functions, and investigate  what kind of solutions one can expect in that case.  Once a generalised Feynman-Kac formula is obtained in the irregular/distributional case, then new tools and methods can be used to investigate irregular physical  phenomena described by (S)PDEs with distributional coefficients. In particular, PDEs like  \eqref{eq: PDE backward intro} with irregular fields $b$ have been considered as models of transport of passive scalars in turbulent fluids (like the  Kraichnan model \cite{kraichnan68}). In recent years the Kraichnan model has been researched by physicists also when the velocity field is a stochastic process, see e.g.\ \cite{pagani15} or \cite{gawedzki08} and references therein. An example of $b$ that we can treat in this paper is the formal gradient of the realization of some random field (like fractional Brownian noise cut at infinity, but one could consider also other fields not necessarily Gaussian so long as their realizations are $\alpha$-H\"older continuous with $\alpha>1/2$).

 In this paper we are indeed able to derive a Feynman-Kac formula that links the PDE \eqref{eq: PDE backward intro} with the forward-backward equations \eqref{eq: forward-backward SDE transformed}  and \eqref{eq: forward-backward SDE}, but our starting point is the solution of the PDE. Hence we use our knowledge on the PDE to infer results on the FBSDE. This is only partially satisfactory   if one argues that using  FBSDEs to solve PDEs is more interesting than the vice versa, but nevertheless the  link provides new stochastic tools to represent and study such turbulent PDEs. For example numerical methods to solve FBSDEs could be employed to find the numerical solution of the PDEs using the Feynman-Kac formula illustrated in this paper.
Indeed there is a line of research that exploits this connection and uses numerical solutions of BSDEs to infer solutions of PDEs  (for a recent work on this see e.g.\  \cite{kharroubi-etal.2015}).

\vspace{7pt}
\emph{Novelty and main results.} 
The present paper is the first to deal with FBSDEs like \eqref{eq: forward-backward SDE transformed} with distributional coefficients  appearing in the driver,  both in the one-dimensional and in the multidimensional case. Because of the lack of literature on this topic, the first challenge we face is to define a suitable  notion of  solution for the backward component of the FBSDE (see Definition \ref{def: virtual solution BSDE} of  virtual-strong solution). Once this is done, the next challenge is to investigate existence and uniqueness of the solution. To do so, we introduce a transformation --which in some sense can be regarded as the analogous for BSDEs of a Zvonkin transformation for SDEs-- and rewrite the original BSDE as an auxiliary backward SDE which can be treated with classical methods, see equation \eqref{eq: BSDE virtual equivalent}. For the auxiliary BSDE it is then possible to show  \emph{existence and uniqueness of a strong solution}, which leads to the same result for the original BSDE \eqref{eq: forward-backward SDE transformed} by transforming back the equation, see Theorem \ref{theor: exist uniq of backward virtual solution}. 
It is worth stressing the fact  that the solution we find is a strong type of solution (and not weak, i.e.\ not of martingale type like in  \cite{russo_wurzer15}). This is possible in the first place because the forward equation here is a Brownian motion and not a solution of a martingale problem.

The second main result in this paper is a \emph{non-linear Feynman-Kac representation formula} that links the PDE \eqref{eq: PDE backward intro} and the FBSDE \eqref{eq: forward-backward SDE transformed}  (see Theorem \ref{theor: Feynman-Kac construct virtual solution} and Theorem \ref{theor: Feynman-Kac construct mild solution}). To show this, we consider smooth approximations of $b$ and  related solutions to the FBSDE and the PDE, and then take the limit.  This requires  various uniform bounds on the smoothed solutions of the  PDE \eqref{eq: PDE backward intro} and of auxiliary PDE \eqref{eq: PDE aux} (see Sections \ref{sc: semi-linear PDE} and \ref{sc: auxiliary PDE and BSDE}).  Indeed the study of PDE \eqref{eq: PDE backward intro} is crucial in this paper because its solution  is used to define \emph{virtual solutions} for both FBSDEs systems \eqref{eq: forward-backward SDE transformed} and \eqref{eq: forward-backward SDE}, as illustrated in Definition \ref{def: virtual solution BSDE} and Definition \ref{def: virtual-weak solution FBSDE}. 
 We solve the semi-linear PDE \eqref{eq: PDE backward intro} by looking for mild solutions using a fixed-point argument. This is the same idea applied in \cite{flandoli_et.al14, issoglio13} where linear PDEs of transport-diffusion type with  distributional coefficients analogous to $b$ have  been studied. The novelty here is the non  linear term $f$, and for this we require Lipschitz continuity properties.  
 Moreover there is a  delicate issue about $f$ that we want to mention at this point, namely the need to match the two set-ups in which the PDE and the FBSDE naturally live, which clearly reflects on the assumptions on the coefficients.
 The former (PDE) is solved as an infinite-dimensional equation, in particular the solution as a function of time takes values in a Sobolev space and so the Lipschitz continuity required for the non-linearity $f$ must be set up in terms of Sobolev spaces (see Assumption \ref{ass: f in sobolev space}). The latter (FBSDE) is set-up in $\mathbb R^d$ and thus assumptions on the coefficients (including $f$) cannot be made in the Sobolev space, but are written in $\mathbb R^d$ instead (see Assumption \ref{ass: f in Rd}). Thus some care is needed to match the two settings and this is explained in Remark \ref{rm: link between Assumptions}. 

The final main result is about the FBSDE  \eqref{eq: forward-backward SDE}. This system is, in some sense, the generalization to multi dimensions of the BSDE studied in \cite{russo_wurzer15}, but with deterministic terminal time.  The system is decoupled and the forward equation is solved first. Here we study the forward equation with different techniques than in \cite{russo_wurzer15}, in particular we invoke the results  found in \cite{flandoli_et.al14} about  SDEs with distributional coefficients which can be applied to the forward component of \eqref{eq: forward-backward SDE}.  
The forward solution $ X_s^{t,x}$ is then used together with standard arguments to find a virtual-weak solution $(X^{t,x},  Y^{t,x}, Z^{t,x} )$ to \eqref{eq: forward-backward SDE},  see Theorem \ref{thm: exist! virtual-weak sol}.
Finally in  Theorem \ref{thm: Feynman-Kac formula} we give a stochastic representation  $(X^{t,x},  Y^{t,x}, Z^{t,x} )= (X^{t,x},   u(\cdot, X^{t,x}), \nabla u(\cdot,  X^{t,x} ))$ of the solution to the FBSDE \eqref{eq: forward-backward SDE} using the  mild solution $u$ of the PDE \eqref{eq: PDE backward intro}.

For system \eqref{eq: forward-backward SDE} we do not find strong solutions  but only weak solutions, because  the solution of the forward equation is of weak type.
We refer the reader to Section \ref{ssc: heuristc comments} for some extended and heuristic comments on the link between \eqref{eq: forward-backward SDE transformed}  and \eqref{eq: forward-backward SDE}, and for open questions.

\vspace{7pt}
\emph{Organization of the paper.}
The paper is organised as follows: In Section \ref{sc: preliminaries} we introduce the notation and recall some useful results; In Section  \ref{sc: semi-linear PDE} we study the PDE \eqref{eq: PDE backward intro} and find a unique mild solution with related smoothness properties. In Section \ref{sc: forward-backward SDE transformed} we introduce the notion of virtual-strong solution for backward SDE \eqref{eq: forward-backward SDE transformed} and show that a unique virtual-strong solution exists. Moreover we establish the non-linear Feynman-Kac formula for \eqref{eq: PDE backward intro} and \eqref{eq: forward-backward SDE transformed}. Finally in Section \ref{sc: forward-backward SDE} we recall the notion of virtual solution for the forward SDE in \eqref{eq: forward-backward SDE}, we show existence and uniqueness of a virtual-weak solution to \eqref{eq: forward-backward SDE} and we provide its explicit representation by means of a non-linear Feynman-Kac formula.

Throughout the paper the constants $C$ and $c$ can vary from line to line.

\section{Preliminaries}\label{sc: preliminaries}
Here we recall some known facts, for more details see \cite[Section 2.1]{flandoli_et.al14} and references therein.  

Let $(P(t), t\geq0)$ be the heat semigroup on the space of $\mathbb R^d$-valued Schwartz functions  $\mathcal S(\mathbb R^d)$   generated by $\frac12 \Delta$, that is the semigroup with kernel $p_t(x)= \frac{1}{(2\pi t)^{d/2}} \exp \left( -\frac{|x|^2}{2t}\right) $, where $|\cdot|$ denotes the Euclidean norm in  $\mathbb R^d$. The semigroup extends to the space of Schwartz distributions {$\mathcal S'(\mathbb R^d)$} by duality, and in particular it maps any {$L^p(\mathbb R^d)$} into itself for $1<p<\infty$. This restriction to {$L^p(\mathbb R^d)$}, denoted  by $(P_p(t), t\geq 0)$, is a bounded analytic semigroup (see \cite[Theorems 1.4.1, 1.4.2]{davies89}). Let  $A_p:= I-\frac12 \Delta$, then $-A_p$ also generates a bounded analytic semigroup which is given by $e^{-t} P_p(t)$ (i.e.\ with kernel $e^{-t}p_t(x)$). We can define fractional Sobolev spaces as images of fractional powers of $A_p$ (which are well defined for any power $s\in \mathbb R$, see \cite{pazy83}) by {$H^s_p(\mathbb R^d):= A_p^{-s/2}(L^p(\mathbb R^d))$}. These are Banach spaces endowed with the norm $\|u\|_{H^s_p}:=\|A_p^{s/2}u\|_{L^p}$. 
It turns out that these spaces correspond to the domain of fractional powers of $A_p$ and of $-\frac12 \Delta$, that is $D(A_p^{s/2}) =D((-\frac12 \Delta)^{s/2}) = H^s_p(\mathbb R^d) $. Moreover $A^{-\alpha/2}_p$ is an isomorphism between $H^s_p(\mathbb R^d)$ and $H_p^{s+\alpha}(\mathbb R^d)$, for each $\alpha\in \mathbb R$.  $H^s_p(\mathbb R^d; \mathbb R^n)$ are defined as above for each component. For shortness of notation we will sometimes denote them simply by  $H^s_p$ (note that the dimension $n$ could be $d, m$ or $m\times d$ depending on the context).   When we write $u\in H^s_p $ we mean that each component $u_i$ is in  $H^s_p(\mathbb R^d)$. The norm will be denoted with the same notation for simplicity. One can also show that $\nabla: H^{1+\delta}_p\to H^\delta_p$ is a continuous map, so  if $u\in H^{1+\delta}_p$ then $\|\nabla u\|_{H^{\delta}_p}\leq c \|u\|_{H^{1+\delta}_p}$ for some positive constant $c$.

The semigroup $(P_p(t), t\geq 0 )$ is a contraction on the $H^s_p(\mathbb R^d)$-spaces for all $t>0$ and all $s\in \mathbb R$ and moreover it enjoys the following mapping property: for $\delta>\beta\geq 0, \delta+\beta<1$ and $0<t\leq T$ it holds  $P_p(t): H^{-\beta}_p(\mathbb R^d)\to H^{1+\delta}_p(\mathbb R^d)$, in particular we have
\begin{equation}\label{eq: mapping of Pt}
\|P_p(t)w\|_{H_p^{1+\delta}(\mathbb R^d)}\le C  t^{-\frac{1+\delta+\beta}{2}}\|w\|_{H_p^{-\beta}(\mathbb R^d)}
\end{equation}
for $w\in H^{-\beta}_p(\R^d), t>0$, where $C= c e^T$ for some positive constant $c$. This follows from a similar property for the semigroup   $(e^{-t}P_p(t), t\geq 0)$  which is stated in \cite[Lemma 10]{flandoli_et.al14}, see also  \cite[Proposition 3.2]{issoglio13} for the analogous on domains $D\subset \mathbb R^d$.
 
  Here we recall the definition of the {\em pointwise product} between a function and a distribution (see \cite{runst_sickel96}) as we will use it several times in this paper.
Let $g\in \mathcal{S}^\prime(\R^d)$. We choose a function $\psi\in \mathcal S (\R^d) $ such that  $0\le \psi(x)\le 1 $, for every $x\in \R^d$  and 
 \begin{equation*}
 \psi(x)=\left\{
 \begin{array}{ll}	
1, &\quad|x|<1,\\
0, &\quad|x|\ge \frac32.
 \end{array}\right.
 \end{equation*}
 For every $j\in \mathbb{N}$, we consider the approximation $S^j g$ of $g$ as follows:
 \begin{equation*}
 S^jg (x):=\mathcal{F}^{-1} \lt(\psi\lt(\frac{\xi}{2^j}\rt)\mathcal{F}(g)\rt)(x),
 \end{equation*}
 where $\mathcal{F}(g)$ and $\mathcal{F}^{-1}(g)$ are the Fourier transform  and the inverse Fourier transform  of $g$, respectively. 
 The  product $gh$ of $g, h\in \mathcal{S}^\prime(\R^d)$ is defined as 
 \begin{equation} \label{eq: pointwise product}
 gh:=\lim_{j\to\infty}S^j g S^j h,
 \end{equation} 
 if the limit exists in  $\mathcal{S}^\prime(\R^d)$. The convergence of the limit \eqref{eq: pointwise product} in the case we are interested in is given by the following result (for a proof see \cite[Theorem 4.4.3/1]{runst_sickel96}).
 \begin{lemma}\label{lm: pointwise product}
 Let  $g\in H_q^{-\beta}(\mathbb{R}^d)$, $h\in  H_p^{\delta}(\mathbb R^d)$ for $1<p,q<\infty$, $q>\max(p,\frac{d}{\delta})$, $0<\beta<\frac{1}{2}$ and $\beta<\delta$. Then the pointwise product $gh$ is well defined, it belongs to the space $ H_p^{-\beta}(\mathbb R^d)$ and we have the following bound
\begin{equation*} 
 \|gh   \|_{  H_p^{-\beta}(\mathbb R^d)} \leq c\| g\|_{ H_q^{-\beta}(\mathbb{R}^d)}   \cdot  \|  h \|_{ H_p^{\delta}(\mathbb R^d)}.
\end{equation*}
\end{lemma}
 
For the following, see \cite[Section 2.7.1]{triebel78}.  The closures of  $\mathcal S$ with respect to the norms  
$$
\|h\|_{C^{0,0}} := \|h\|_{L^\infty}
$$ 
and 
$$
\|h\|_{C^{1,0}} := \|h\|_{L^\infty} + \|\nabla h\|_{L^\infty} 
$$ 
respectively, are denoted by {$C^{0,0}(\mathbb R^d; \mathbb R^m)$} and {$C^{1,0}(\mathbb R^d; \mathbb R^m)$}. 
For any $\alpha >0$, we  consider the Banach spaces
\begin{align*}
	&C^{0, \alpha}= \{ h \in {C^{0,0}(\mathbb R^d; \mathbb R^m)} :  \|h\|_{C^{0,\alpha}} <\infty \}\\
	&C^{1, \alpha} =\{  h \in {C^{1,0}(\mathbb R^d; \mathbb R^m)} :  \|h\|_{C^{1,\alpha}} <\infty \},
\end{align*}
endowed with the norms
\begin{align*}
	&\|h\|_{C^{0,\alpha}} := \| h \|_{L^\infty}  + \sup_{x\neq y \in \mathbb R^d}  \frac{|h(x)-h(y)|}{|x-y|^\alpha} \\
	& \|h\|_{C^{1,\alpha}} := \| h \|_{L^\infty}  + \| \nabla h \|_{L^\infty} + \sup_{x\neq y \in \mathbb R^d}  \frac{|\nabla h(x)-\nabla h(y)|}{|x-y|^\alpha},
\end{align*}
respectively.

Let $B$ be a Banach space. We denote by $C^{0, \alpha}([0,T]; B)$ the space  analogous to $C^{0, \alpha}$ but with values in $ B$, and its norm by $\| \cdot \|_{ C^{0, \alpha}([0,T]; B)}$. We denote by $C([0,T];B)$ the Banach space of $B$-valued continuous functions and its sup norm by $\|\cdot\|_{\infty,B}$. For $h\in C([0,T],B)$, we also use the family of equivalent norms $\{\|\cdot\|_{\infty,B}^{(\rho)}, \rho\ge 1\}$, defined by
\[
\|h\|_{\infty,B}^{(\rho)}:=\sup_{0\le t\le T} e^{-\rho t} \|h(t)\|_{B}. 
\]  
The usual esssup norm on $L^\infty (0,T; B)$ will also be denoted  by $\|\cdot\|_{\infty, B}$ with a slight abuse of notation. The Euclidean norm in $\mathbb R$, $\R^d$, $\R^m$, and the Frobenius norm in $\R^{m\times d}$ will be denoted by $|\cdot|$.

The following lemma provides a generalization of the Morrey inequality to fractional Sobolev spaces. For the proof we refer to \cite[Theorem 2.8.1, Remark 2]{triebel78}.

\begin{lemma}[Fractional Morrey inequality]\label{lm: fractional Morrey ineq}
	Let $0<\delta< 1 $ and $d/\delta<p<\infty$. If $h\in H^{1+\delta}_p(\mathbb R^d)$ then there exists a unique version of $h$ (which we denote again by  $h$) such that $h$ is differentiable. Moreover $ h\in C^{1,\alpha}(\mathbb R^d)$ with $\alpha= \delta-d/p$ and
	\begin{equation}\label{eq: fractional Morrey ineq}
	\|h\|_{C^{1,\alpha}}\leq c \|h\|_{H^{1+\delta}_p}, \quad  \|\nabla h\|_{C^{0,\alpha}}\leq c \|\nabla h\|_{H^{\delta}_p},
	\end{equation}
	where $c=c(\delta, p, d)$ is a universal constant.
\end{lemma}

\noindent  {\bf Standing Assumption:} Throughout the paper we will make the following standing assumption about the drift $b$ and in particular about the parameters involved. We acknowledge that the set $K(\beta, q)$ is taken from \cite{flandoli_et.al14}.\\ 
{\em Let $\beta\in\left(  0,\frac{1}{2}\right)  $, $q\in\left(  \frac{d}{1-\beta},\frac{d}{\beta}\right)  $. Let the drift $b$ be of the type
		\[
		b\in L^\infty\left ([0,T]; H^{-\beta}_{q}(\R^d; \R^d) \right).
		\]
Moreover for given $\beta$ and $q$ as above we define the  set
		\begin{equation}\label{eq: set admissible kappa}
		K(\beta,q):=\left\{\kappa=(\delta, p): \; \beta< \delta< 1-\beta, \,  \frac d\delta < p < q  \right\}.
		\end{equation}
We always choose $(\delta, p)\in K(\beta,q)$. 
Note that $K(\beta,q)$ is non-empty since $\beta<\frac12 $ and $ \frac d{1-\beta}< q < \frac d\beta $. }

Regarding the  functions $f$ and $\Phi$, we make the following parallel sets of assumptions. This is because the PDE   is set (and solved) using fractional Sobolev spaces, whereas the BSDE is typically set in $\mathbb R^d$. We discuss the link and implications of these two sets of Assumptions in Remark \ref{rm: link between Assumptions} below. Afterwards, we also give examples of possible $f$. Note that the notation for  $f$ is the same, even though the function is in principle different in the two sets of assumptions. 
\begin{assumption} \label{ass: f in Rd} \ 
\begin{itemize}
\item {$\Phi:\mathbb R^d\to \mathbb R^m$} is such that $\Phi\in H^{1+\delta +2\gamma}_p$ for some $\gamma<\frac{1-\delta-\beta}{2}$; 
\item   {$f:[0,T]\times\mathbb R^d \times \mathbb R^m\times \mathbb R^{m\times d} \to \mathbb R^m$} is continuous in $(x, y,z)$ uniformly in $t$, and Lipschitz continuous in $(y,z)$ uniformly in $t$ and $x$, i.e.\ $\vert f(t,x,y,z)-f(t,x, y', z')\vert\leq L(\vert y-y'\vert+ \vert z-z'\vert )$ for any   {$y,y' \in \mathbb R^m$ and $z, z'\in \mathbb R^{m\times d}$}; 
\item  $sup_{t,x} |f(t,x,0,0)| \leq  C$ and $\sup_{t\in[0,T]} \int_{\mathbb R^d}|f(t,x,0,0)|^p \mathrm dx \le C$. 
\end{itemize}
\end{assumption}  
  
\begin{assumption}\label{ass: f in sobolev space}\  
\begin{itemize}
\item  $\Phi\in H^{1+\delta +2\gamma}_p(\R^d; \R^m)$ for some $\gamma<\frac{1-\delta-\beta}{2}$;
\item $f:[0,T]\times H^{1+\delta}_p (\R^d;\R^m)  \times H^{\delta}_p (\R^d; \R^{m\times d}) \to H^{0}_p  (\R^d;\R^m)  $  is Lipschitz continuous in the second and third variable uniformly in $t$, that is, there exists a positive constant $L$ such that for any $u_1, u_2\in H^{1+\delta}_p  $ and $v_1, v_2 \in H^{ \delta}_p $ then $$\|f(t,u_1, v_1)- f(t, u_2, v_2)\|_{H^{0}_p}\leq L\left ( \|u_1-u_2\|_{H^{1+\delta}_p} +  \|v_1-v_2\|_{H^{ \delta}_p} \right) ;$$
\item $sup_{t,x} |f(t,x,0,0)| \leq C$ and  $\sup_{t\in[0,T]} \|f(t,0,0)\|_{H^{0}_p}\le C$, where $0$ here denotes the constant zero function.
\end{itemize}
\end{assumption}
\textbf{Notation:} In Assumption \ref{ass: f in sobolev space} the functional $f$ is a function of time $t$ and of two other functions, often denoted by $u$ and $v$ (or $u$ and $\nabla u$). In this paper we   write $f(t, u,v)$, or  $f(t, \cdot, u, v)$, or also $f(t, \cdot, u(\cdot), v(\cdot))$, and this is  an element of the space $H^{1+\delta}_p$ by Assumption \ref{ass: f in sobolev space}. 

\begin{remark}\label{rm: link between Assumptions}
\begin{itemize}
\item By applying the Fractional Morrey inequality we see that $\Phi\in C^{1,\alpha}$ with $\alpha = \delta+2\gamma -\frac dp>0$. This implies in particular that $\Phi$ is bounded and continuous. Note that the latter would be the standard assumption on the terminal condition $\Phi$ when solving the BSDE, but our setting to solve the PDE requires that $\Phi$ is an element of  fractional Sobolev spaces and  we will use the fact that Assumption \ref{ass: f in Rd} implies Assumption \ref{ass: f in sobolev space}, as illustrated below. 
\item Assumption   \ref{ass: f in Rd} implies Assumption \ref{ass: f in sobolev space}. Indeed take $f$ according to Assumption \ref{ass: f in Rd}. Then we can define the functional $\bar f$ as follows $\bar f (t, u, v) (\cdot):= f(t,\cdot, u(\cdot), v(\cdot) )$ for   $u\in H^{1+\delta}_p$ and $ v\in  H^{ \delta}_p  $.  
The first and third bullet points of  Assumption \ref{ass: f in sobolev space}  are obvious.  
 The second bullet point can be proven as follows.
First we show that for $(t,u,v)\in [0,T]\times H^{1+\delta}_p \times H^{\delta}_p $ then  $\bar f (t, u, v)\in H^0_p$. Indeed we have
\begin{align*}
\phantom{=}&\int_{\mathbb R^d} |\bar f(t,u,v)(x)|^p \mathrm dx 
= \int_{\mathbb R^d} | f(t, x,u(x),v(x))|^p \mathrm dx \\
\leq & c \int_{\mathbb R^d} | f(t, x,u(x),v(x))- f(t,x,0,0)|^p \mathrm dx \\
& + c \int_{\mathbb R^d} | f(t,x,0,0)|^p \mathrm dx \\
\leq & c L^p (\|u\|^p_{H^{1+\delta}_p}+ \|v\|^p_{H^{\delta}_p}) +\sup_{0\leq t\leq T} \|f(t, 0, 0)\|_{H^{0}_p} <\infty,
\end{align*}
where the constant $c$ depends on $p$.

Now with similar calculations one can prove that given any  $u, u'\in  H^{1+\delta}_p $ and $v, v'\in  H^{\delta}_p $ it holds
$$\|\bar f (t, u, v) - \bar f (t, u', v') \|_{H^0_p} \leq cL \left( \|u-u' \|_{H^0_p} +\|v-v' \|_{H^0_p}\right),$$
where the constant $c$ depends on $p$, and $L$ is the Lipschitz constant for $f$.  
\end{itemize}
\end{remark}

\textbf{Example}.
\begin{itemize}
\item An easy case  is the class of functions $f$  linear in $(y,z)$, for example $f(t,x,y,z)= c(t)\cdot(y+z)+d(x)$, where $t\mapsto c(t)$ is continuous on $[0,T]$ and $x\mapsto d(x)$ is bounded in $\mathbb R^d$ and $L^p(\mathbb R^d)$-integrable, for example $d(x) =  e^{-|x|^2}$.  In this case we would have $\bar f(t, u, v) = c(t)\cdot( u + v) + d $.
\item A non-linear example is given by $f(t,x,y,z)= c(t)\cdot \sin(y+z)+d(x)$, where $c$ and $d$ are as above. Then we would get $\bar f(t, u, v) = c(t)\cdot\sin ( u + v) +d$, which is Lipschitz continuous in $(u,v)$ and bounded at $0$ uniformly in $(t,x)$.
\end{itemize}

\section{The semi-linear PDE}\label{sc: semi-linear PDE}
In this section we analyse the PDE \eqref{eq: PDE backward intro} and obtain several bounds for its solution and for the mollified version. We refer the reader to \cite{hinz_issoglio_zaehle13, issoglio13} for results on different (S)PDEs obtained using similar techniques, and \cite{issoglio_zaehle15} for the general case of linear equations in metric measure spaces. 

\subsection{Existence and uniqueness of a mild solution}\label{ssc: existence and uniqueness mild sol}
We recall the PDE below for ease of reading:
\begin{equation}\label{eq: PDE backward Rd}
\left\{\begin{array}{ll}
u_t(t,x) +L^b u (t,x)  + f(t,x,u(t,x), \nabla u(t,x))=0, \\
u(T,x)=\Phi(x),\\
\forall (t, x)\in[0,T] \times \mathbb R^d.
\end{array}
\right.
\end{equation}
Here the operator $L^b u = \frac 12 \Delta u +\nabla u b$ is defined component by component by $ (L^b u)_i  (t,x) = \frac12 \Delta u_i(t,x) +  \nabla u_i(t,x)b(t,x)$ for all $i=1, \ldots, d$.
The peculiarity of this PDE is that it  involves a distributional coefficient $b $ and in particular its product with $\nabla u$. The meaning we give to this product makes use of the pointwise product recalled in Section \ref{sc: preliminaries}. We follow the study of a similar equation from the first author in \cite{issoglio13}. Here the novelty is that the PDE is non-linear, with the extra term $f$ appearing. We are going to look for mild solutions, hence the following definition is in order.

\begin{definition}
	A mild solution of \eqref{eq: PDE backward Rd} is an element $u$ of $C([0,T], H^{1+\delta}_p)$ which is a  solution of the following integral equation
	\begin{align} \nonumber
		u\left(  t\right)  =& P_p(T-t) \Phi	+ \int_{t}^{T}P_p (r-t)\left( 
		\nabla u\left(  r\right)  b\left(  r\right)\right) dr\\ \label{eq: mild solution PDE}
		&+\int_{t}^{T}P_p\left(  r-t\right) f \left(r, u(r) , \nabla u (r) \right)dr,
	\end{align}
	where $(P_p(t), t\geq0)$ is the semigroup generated by $\frac12 \Delta$ and recalled in Section \ref{sc: preliminaries}.
\end{definition}

To solve the PDE \eqref{eq: PDE backward Rd} we will use a fixed point argument in equation  \eqref{eq: mild solution PDE} and for that we need $f$ to be an element of a fractional Sobolev space as function of $x$  and further to be Lipschitz continuous in such space: this is what is stated in Assumption \ref{ass: f in sobolev space}.

\begin{theorem}\label{thm: PDE backward}
	Suppose that Assumption \ref{ass: f in sobolev space}  holds. Then there exists a unique mild solution $u\in C([0,T], H^{1+\delta}_p)$ of \eqref{eq: PDE backward Rd}. 
\end{theorem}

\begin{proof}
	The idea of the proof is   similar to the proof of \cite[Theorem 3.5]{issoglio13} and \cite[Theorem 14]{flandoli_et.al14}: We look for a fixed point in $C([0,T], H^{1+\delta}_p)$, in particular we show that the mapping defined by the right-hand side of \eqref{eq: mild solution PDE} is a contraction by using the family of equivalent norms $\|\cdot\|_{\infty, H^{1+\delta}_p}^{(\rho)}$.
	
	To this aim, we rewrite the mild solution in a forward form for $\bar u(t)=  u(T-t)$. We get 
	\begin{align} \nonumber
	\bar u\left(  t\right) =  & P_p(t) \Phi \\\label{eq: mild solution PDE v}
	& + \int_{0}^{t}P_p (t-r) \left( 
	\nabla \bar u\left(  r\right)b\left(T- r\right)  + f \left(T-r, \bar u(r), \nabla \bar u (r)  \right)\right)dr\\\nonumber
	 = & P_p(t) \Phi+ \int_{0}^{t}P_p (t-r) \left(
	 \nabla \bar u\left(  r\right)\bar{b}\left(  r\right) + \bar{f} \left(r, \bar u(r), \nabla\bar u (r)  \right)\right)dr,
	\end{align}
    where $\bar{b}(r)=b(T-r)$ and $\bar{f}(r,\bar u (r))=f(T-r,\bar u(r), \nabla \bar u(r))$. Since $\bar{b}, \bar{f}, \bar u$ and $b, f, u$ share the same regularities in $r$, with a slight abuse of notations, in the following we still write $b, f$ and $u$ instead of $\bar{b}, \bar f$ and $\bar{u}$.  
	
	If we denote by $I_t(u)$ the right-hand side of \eqref{eq: mild solution PDE v}, then we need to control the norm $\|I(u_1)-I(u_2)\|^{(\rho)}_{\infty, H^{1+\delta}_p}$ for any $u_1, u_2\in C([0,T], H^{1+\delta}_p)$, which is the sum of three terms: One with the  initial condition, one term with $b$ and one term with $f$. The initial condition $P_p(t)\Phi$ belongs to $ H^{1+\delta}_p$ since $\Phi\in H^{1+\delta+2\gamma}_p \subset H^{1+\delta}_p$ and the semigroup is a contraction on $H^{1+\delta}_p$. The term including the distributional coefficient $b$ can be treated exactly like in \cite[Theorem 3.4]{issoglio13} because the pointwise product is linear. One  gets the bound 
	\begin{align*}
		&\left \| \int_0^\cdot  P_p(\cdot-r) \left(( \nabla u_1(r)-\nabla u_2(r))  b(r) \right)  \mathrm dr \right  \|^{(\rho)}_{\infty, H^{1+\delta}_p} \\
		\leq  &C  \rho^\frac{\delta+\beta-1}{2} \|b\|_{\infty ,H^{-\beta}_p } \|u_1-u_2\|^{(\rho)}_{\infty,H^{1+\delta}_p},
	\end{align*}
	which is finite and the constant $C  \rho^\frac{\delta+\beta-1}{2} $ tends to zero as $\rho\to \infty$ since $\delta+\beta-1<0$ by assumption on the parameters. 
	
	The third term involves $f$ and is estimated  using the Lipschitz regularity of $f$   and the mapping property \eqref{eq: mapping of Pt} of $P_p(t)$ with $\beta=0$. We get 
\begin{equation*} 
	\begin{aligned} 
		&\left \|\int_0^\cdot  P_p(\cdot-r)   f \left(r, u_1(r), \nabla u_1(r)  \right)  \mathrm dr  -\int_0^\cdot  P_p(\cdot-r)   f \left(r, u_2(r), \nabla u_2(r)   \right)  \mathrm dr \right  \|^{(\rho)}_{\infty,H^{1+\delta}_p} \\ 
		\leq & \sup_{0\leq t \leq T}e^{-\rho t} \int_0^t \| P(t-r) \big(  f \left(r, u_1(r), \nabla u_1(r)   \right) -   f \left(r, u_2(r), \nabla u_2(r)   \right)\big) \|_{H^{1+\delta}_p} \mathrm dr \\
		\leq & \sup_{0\leq t \leq T}e^{-\rho t} \int_0^t  C r^{-\frac{1+\delta}{2}} \|  f \left(r, u_1 (r) , \nabla u_1(r)  \right) -   f \left(r, u_2(r) , \nabla u_2(r)  \right) \|_{H^{0}_p} \mathrm dr \\
		\leq & C \sup_{0\leq t \leq T} \int_0^t e^{-\rho (t-r)}    e^{-\rho r} r^{-\frac{1+\delta}{2}} L\left ( \|  u_1(r) -u_2(r)   \|_{H^{1+\delta}_p} +\| \nabla  u_1(r) - \nabla u_2(r)   \|_{H^{\delta}_p}  \right) \mathrm dr\\
		\leq & 2 C  \|  u_1 -u_2   \|^{(\rho)}_{\infty, H^{1+\delta}_p}  \sup_{0\leq t \leq T} \int_0^t e^{-\rho (t-r)} r^{-\frac{1+\delta}{2}} \mathrm dr\\
		\leq & C  \rho^{\frac{\delta-1}{2}} \| u_1-u_2   \|^{(\rho)}_{\infty,H^{1+\delta}_p},
		\end{aligned}
	\end{equation*}
	where in the second to last inequality we used the definition of $\rho$-equivalent norm and the continuity of $\nabla:H^{1+\delta}_p\to H^\delta_p$.
	Note that again the exponent of $\rho$ is negative since $\delta <1$ by assumption.
	Thus for $\rho$ large enough we have $$ \|I(u_1)-I(u_2)\|^{(\rho)}_{\infty,H^{1+\delta}_p} \leq C \|u_1-u_2\|^{(\rho)}_{\infty, H^{1+\delta}_p},$$ where $C<1$ does not depend on $u_1 $ and $u_2$. Hence by Banach's contraction principle there exists a unique solution $u\in  C([0,T], H^{1+\delta}_p)$. 
\end{proof}

\begin{remark}\label{rem: more on solution of PDE}
	Thanks to the choice of the parameters $\delta$ and $p$ in $K(\beta,q)$ (which is always possible since $p>d/\delta$, see   \cite{flandoli_et.al14} for more details) and to Lemma \ref{lm: fractional Morrey ineq}, we have the embedding of $H^{1+\delta}_p$ in $C^{1, \alpha}$, where $\alpha=\delta-d/p$. So for each $t\in[0,T]$, the solution $u(t)$ as a function of $x$ is in fact bounded, differentiable and the first derivative is H\"older continuous, $u(t)\in C^{1, \alpha}$. 
\end{remark}

We will use \cite[Proposition 11]{flandoli_et.al14}  several times in this paper.  We recall it here for the reader's convenience.

\begin{proposition} 
	\label{prop: 11inFlIsRu}
	Let $h\in L^{\infty}\lt([0,T]; H_p^{-\beta}\rt)$ and $g:[0,T]\to H_p^{-\beta}$ for $\beta\in \R$ be defined as
	\[
	g(t)=\int^t_0 P_p(t-r) h(r)\d r.
	\]
	Then $g\in C^{0,\gamma}\lt([0,T];H_p^{2-2\ep-\beta}\rt)$ for every $\ep>0$ and $\gamma\in (0,\ep)$. Moreover, we have
	\begin{equation}
	\label{eq: boundFl}
		\|g(t)-g(s)\|_{H_p^{2-2\ep-\beta}}
		\le C (t-s)^\gamma\lt((t-s)^{\ep-\gamma}  +s^{\ep-\gamma}\rt)\|h\|_{\infty,H_p^{-\beta}}.
	\end{equation}
\end{proposition}

The proof of bound (\ref{eq: boundFl}) can be found in the proof of \cite[Proposition 11]{flandoli_et.al14}.

Additionally we can show the following lemma.

\begin{lemma}
	The mild solution $u$ of \eqref{eq: PDE backward Rd} 
	is  H\"older continuous in time of any order  $\gamma<\frac{1-\delta-\beta}{2}$, that is, $u\in C^{0,\gamma}([0,T];$ $H^{1+\delta}_p)$.
\end{lemma} 

\begin{proof}
	This is done using the results of Proposition \ref{prop: 11inFlIsRu}  with $\varepsilon=\frac{1-\delta-\beta}{2}$ and noting that  $P_p(\cdot)\Phi $ is $\gamma$-H\"older continuous if $\Phi\in H_p^{1+\delta+2\gamma}$, with $2\gamma<1-\delta-\beta$.
\end{proof}

\subsection{Uniform bounds on mollified mild solution}\label{ssc; uniform bounds mollified mild sol}

In the next sections we will make use of an approximating sequence $b^n$ in place of $b$. We therefore need to describe its effect on the solution of the PDE \eqref{eq: PDE backward Rd} where the coefficient $b$ is replaced by a  coefficient $b^n$, that is 
\begin{equation}\label{eq: PDE backward for u^n}
\left\{
\begin{array}{l}
  u^n_t(t,x) +  L^{b^n} u^n(t,x) + f(t,x,u^n(t,x),\nabla u^n(t,x))=0, \\
  u^n(T,x)=\Phi(x),   \\
  \forall (t,x)\in[0,T]\times \R^d,
\end{array}\right.
\end{equation}
 where $ L^{b^n} u^n(t,x) := \frac12 \Delta u^n(t,x) +  \nabla u^n(t,x)b^n(t,x)$ is the analogue of $L^b$.

If {$b^n$ is smooth, for example} $b^n\in C([0,T];C_b^1(\R^d;\R^d))$ (bounded with bounded first derivatives), then $u^n $ is a classical solution and it coincides with the mild solution found in Theorem \ref{thm: PDE backward}.  We will use this fact for example in the proof of Theorem \ref{theor: Feynman-Kac construct virtual solution}. In what follows we state and prove some continuity results which hold also for $b^n$ non-smooth.  

\begin{lemma}\label{lm: continuity of u wrt approximation}
Let Assumption \ref{ass: f in sobolev space} hold, and let $b^n\to b$ in $L^\infty\left ([0,T]; H^{-\beta}_{q} \right)$. Then
	\begin{itemize}
		\item[(i)] $u^n\to u$ in $ C([0,T];H^{1+\delta}_{p})$ and there exists a constant $C$ independent of $n$ such that
		$$
		\|u^n-u\|_{\infty,H_p^{1+\delta}}\le C\|b^n-b\|_{\infty, H_{q}^{-\beta}}.
		$$
		\item[(ii)]
		$u^{n}\rightarrow u$ and  $\nabla u^{n}\rightarrow\nabla u$ uniformly on $\left[  0,T\right]  \times\mathbb{R}^{d}$.
		
	\end{itemize}
	
\end{lemma}

\begin{proof}
	(i) By similar calculations as in Theorem \ref{thm: PDE backward}
	and by adding and subtracting $b^n(r)\nabla u(r)$ 	 we have
	\begin{align*}
		&\|u-u^n\|^{(\rho)}_{\infty,H_p^{1+\delta}}=\sup_{t\in[0,T]} e^{-\rho t} \|u(t)-u^n(t)\|_{H_p^{1+\delta}}\\
		\le &\sup_{t\in[0,T]} e^{-\rho t} \bigg(\int^t_0 \|P_p(t-r)( \nabla u^n(r)b^n(r)-u(r)b(r))\|_{H_p^{1+\delta}}\d r \\
		&+\int^t_0 \|P_p(t-r)(f(r,u^n(r),\nabla u^n(r))-f(r,u(r), \nabla u(r)))\|_{H_p^{1+\delta}}\d r\bigg) \\
		\le &\sup_{t\in[0,T]} \bigg(C\int^t_0e^{-\rho(t-r)}(t-r)^{-\frac{1+\delta+\beta}{2}}e^{-\rho r}\|b^n(r) \|_{H_q^{-\beta}} \|u^n(r)-u(r)\|_{H_p^{1+\delta}}\d r\\
		&+ C\|b^n-b\|_{\infty,H_q^{-\beta}}\int^t_0e^{-\rho(t-r)}(t-r)^{-\frac{1+\delta+\beta}{2}}e^{-\rho r}\|u(r)\|_{H_p^{1+\delta}}\d r\bigg)\\
		&+\sup_{t\in[0,T]} e^{-\rho t}\int^t_0 r^{-\frac{1+\delta}{2}}\|f(r,u^n(r),\nabla u^n(r))-f(r,u(r),\nabla u(r))\|_{H_p^{0}}\d r\\
		\le &C\|b\|_{\infty,H_q^{-\beta}}\|u^n-u\|_{\infty, {H_p^{1+\delta}}}^{(\rho)} \rho^{\frac{\delta+\beta-1}{2}}+C\|b^n-b\|_{\infty, {H_p^{1+\delta}}}\|u\|_{{H_p^{1+\delta}}}^{(\rho)}\rho^{\frac{\delta+\beta-1}{2}}\\
		&+C\|u^n-u\|_{\infty, {H_p^{1+\delta}}}^{(\rho)}\rho^{\frac{\delta-1}{2}}.
	\end{align*}
	Therefore there exists a $\rho$ big enough so that
	$$
	1-C\lt(\rho^{\frac{\delta+\beta-1}{2}}+\rho^{\frac{\delta-1}{2}}\rt)>0.
	$$
	Hence for such $\rho$,
	\begin{align*}
		\|u-u^n\|^{(\rho)}_{\infty,H_p^{1+\delta}}
		\le &\frac{C\|u\|_{H_p^{1+\delta}}^{(\rho)}\rho^{\frac{\delta+\beta-1}{2}}}{1-C\lt(\rho^{\frac{\delta+\beta-1}{2}}+\rho^{\frac{\delta-1}{2}}\rt)}\|b^n-b\|_{\infty, {H_p^{1+\delta}}}\\
		=&C\|b^n-b\|_{\infty, {H_p^{1+\delta}}}.
	\end{align*}
	
	Part (ii) follows from part (i) and by  the Fractional Morrey inequality (Lemma \ref{lm: fractional Morrey ineq}).
\end{proof}

\begin{lemma}\label{lm: uniform bound u^n}
Let Assumption \ref{ass: f in sobolev space} hold and let $b^n$ be   such that $b^n\to b$ in $L^\infty (0,T;H_{q}^{-\beta})$.  
	The mild solution $u^n$ of \eqref{eq: PDE backward for u^n} 
	is  H\"older continuous in time of any order  $\gamma<\frac{1-\delta-\beta}{2}$, that is, $u^n\in C^{0,\gamma}([0,T];$ $H^{1+\delta}_p)$. Moreover, we have the uniform bound:
	\begin{equation}
	\label{ineq: Cgamma norm bound}
	\|u^n\|_{C^{0,\gamma}([0,T];H^{1+\delta}_p)}\le C
	\end{equation} 
	for some $C$ independent of $n$.
\end{lemma} 
\begin{proof}
	We recall that 
	\begin{equation}\label{eq: Cgamma norm}
	\begin{aligned}
	\|u^n\|_{C^{0,\gamma}([0,T]; H^{1+\delta}_p)}
	=&\sup_{0\le t\le T}\|u^n(t)\|_{H^{1+\delta}_p}\\& +\sup_{0\le s<t\le T} \frac{\|u^n(t)-u^n(s)\|_{H_p^{1+\delta}}}{|t-s|^\gamma}.
	\end{aligned}
	\end{equation}
		By Lemma \ref{lm: continuity of u wrt approximation}, the first term on the right-hand side of (\ref{eq: Cgamma norm}) is bounded by 
	\[
	\|u^n\|_{\infty,H_p^{1+\delta}}\le C\|u\|_{\infty,H_p^{1+\delta}},
	\]
	where the constant $C$ is independent of $n$. 
	To bound the second term, let us consider the difference $u^n(t)-u^n(s)$ as the sum of three terms
	\[
	(P_p(t)\Phi-P_p(s)\Phi) + (g_1^n(t)-g^n_1(s)) +(g_2^n(t)-g_2^n(s)),
	\]
	where
	\[
	g_1^n(t)=\int^t_0P_p({t-r}) \nabla u^n(r) b^n(r)\d r
	\]
	and 
	\[
	g_2^n(t)=\int^t_0P_p({t-r}) f(r,u^n(r),\nabla u^n(r))\d r.
	\]
	
	Observe that since $\Phi\in H^{1+\delta+2\gamma}_p$, then $A^{\frac{1+\delta}{2}}\Phi\in H^{2\gamma} $ hence it belongs to $D(A^\gamma)$ and so does $P_p(s)A^{\frac{1+\delta}{2}}\Phi$. We have 
	\begin{equation*}
		\begin{aligned}
			&\|P_p(t)\Phi-P_p(s)\Phi\|_{H_p^{1+\delta}}\\
			\le &C  \|(P_p({t-s})-I)P_p(s) A^{\frac{1+\delta}{2}}\Phi\|_{H_p^{0}}\\
			\le & C(t-s)^\gamma \|P_p(s)A^{\frac{1+\delta}{2}+\gamma}\Phi\|_{H_p^{0}}\\
			\le & C(t-s)^\gamma,
		\end{aligned}
	\end{equation*} 
	where we have used the fact that for any $\phi \in D(A^\gamma )$ then $\|P_t\phi-\phi\|_{H_p^0}\leq C_\gamma t^\gamma \|A^\gamma \phi\|_{H^0_p}$.
	Observe also that for $\ep>0$ such that $1+\delta\le 2-2\ep-\beta, i.e., \ep\le\frac{1-\delta-\beta}{2}$, we have, 	for $i=1, 2$,
	\begin{equation}\label{eq: bound gi}
	\|g_i^n(t)-g_i^n(s)\|_{H^{1+\delta}_p}\le \|g_i^n(t)-g_i^n(s)\|_{H_p^{2-2\ep-\beta}}.
	\end{equation}
    Moreover, for fixed $r\in[0,T]$, we have
	\begin{equation*}
		\begin{aligned}
			\|\nabla u^n(r) b^n(r)\|_{H_p^{-\beta}}\le& C \|b^n(r)\|_{H_q^{-\beta}} \|\nabla u^n(r)\|_{H_p^\delta}\\
			\le & C\|b^n\|_{\infty,H_q^{-\beta}} \|u^n\|_{\infty,H_p^{1+\delta}}.
		\end{aligned}
	\end{equation*}
	Hence by Proposition \ref{prop: 11inFlIsRu} applied to $g_1^n$ and using \eqref{eq: bound gi} we get
	\begin{equation*}
		\begin{aligned}
			\|g_1^n(t)-g_1^n(s)\|_{H_p^{1+\delta}}
			\le& C (t-s)^{\gamma} \lt((t-s)^{\ep-\gamma} + s^{\ep-\gamma}\rt)\| {\nabla u^nb^n}\|_{\infty,H_p^{-\beta}}\\
			\le & C (t-s)^{\gamma} \lt((t-s)^{\ep-\gamma} + s^{\ep-\gamma}\rt),
		\end{aligned}
	\end{equation*}
	where $C$ is independent of $n$ because $u^n\to u$ in $C([0,T],H_p^{1+\delta})$ by Lemma \ref{lm: continuity of u wrt approximation} and $b^n\to b$ in $L^\infty (0,T;H_q^{-\beta})$ by hypothesis.
	
	The difference involving $g_2$ is similar, but instead we use the Lipschitz property of $f$ to get 
	\begin{equation*}
		\begin{aligned}
			&\|f(r,u^n(r),\nabla u^n(r))\|_{H_p^{-\beta}}\\
			\le& \|f(r,u^n(r),\nabla u^n(r))\|_{H_p^{0}}\\
			\le& C  \|f(r,u^n(r),\nabla u^n(r))-f(r,0,0)\|_{H_p^{0}} +C\|f(r,0,0)\|_{H_p^{0}}\\
			\le & C\lt(1+\|u^n(r)\|_{H_p^{1+\delta}} +\|\nabla u^n(r)\|_{H_p^\delta } \rt)\\
			\le & C\lt(1+\|u^n\|_{\infty,H_p^{1+\delta}}\rt),
		\end{aligned}
	\end{equation*}
having also used the fact that $\sup_r \|f(r,0,0)\|_{H_p^{0}} <c$ by Assumption \ref{ass: f in sobolev space}.	Hence by Proposition \ref{prop: 11inFlIsRu} we get
	\begin{equation*}
		\begin{aligned}
			&\|g_2^n(t)-g_2^n(s)\|_{H_p^{1+\delta}}\\
			\le& C (t-s)^\ep \|f(\cdot,u^n,\nabla u^n)\|_{\infty,H_p^{-\beta}} +C(t-s)^\gamma s^{\ep-\gamma}\|f(\cdot,u^n,\nabla u^n)\|_{\infty,H_p^{-\beta}}\\
			\le & C\lt(1+\|u^n\|_{\infty,H_p^{1+\delta}}\rt)\lt((t-s)^\ep+(t-s)^\gamma s^{\ep-\gamma}\rt)\\
			\le & C (t-s)^{\gamma} \lt((t-s)^{\ep-\gamma} + s^{\ep-\gamma}\rt).
		\end{aligned}
	\end{equation*}
	where $C$ is independent of $n$.  
	Putting the three terms together we get 
	\begin{equation*}
		\begin{aligned}
			&\|u^n(t)-u^n(s)\|_{H_p^{1+\delta}}\\
			\le &\|P_p(t)\Phi-P_p(s)\Phi\|_{H_p^{1+\delta}}+
			\|g^n_1(t)-g^n_1(s)\|_{H_p^{1+\delta}} +	\|g^n_2(t)-g^n_2(s)\|_{H_p^{1+\delta}}\\
			\le & C(t-s)^\gamma+2C(t-s)^{\gamma} \lt((t-s)^{\ep-\gamma} + s^{\ep-\gamma}\rt),
		\end{aligned}
	\end{equation*}
	and so the second term on the right-hand side of  (\ref{eq: Cgamma norm}) is bounded by 
	$$
	C+2C\lt((t-s)^{\ep-\gamma}+s^{\ep-\gamma}\rt)\le C(T),
	$$ 
	for $\ep$ such that $\gamma< \ep\le \frac{1-\delta-\beta}{2}$, which is always possible since $2\gamma<1-\delta-\beta$ by assumption.
\end{proof}

Both for $u$ and $u^n$ we have desirable continuity properties and bounds which are uniform in $n$.

\begin{lemma}\label{lm: properties of solution u and u^n of PDE}
Let Assumption \ref{ass: f in sobolev space} hold and let $u$ and $u^n$ be the  solutions of \eqref{eq: PDE backward Rd} and \eqref{eq: PDE backward for u^n} respectively. For $\nu=u$ and $\nu=u^n$, the following properties hold:
	
		For each $t\in[0,T]$ we have $\nu(t)  \in C^{1,\alpha}$ and there exists a positive constant $C$ independent of $n$ such that 
		\begin{equation}\label{eq: sup bound for nu}
		\sup_{0\leq t \leq T}\left( \sup_{x\in \R^d} |\nu(t,x)|\right) \leq C,
		\end{equation}
		and
		\begin{equation}\label{eq: sup bound for nu_x}
		\sup_{0\leq t \leq T}\left( \sup_{x\in \mathbb R^d} | \nabla \nu (t,x)| \right) \leq C.
		\end{equation}
		Moreover, there exists a positive constant $C$ independent of $n$ such that for any $t,s\in[0,T]$ and $x, y \in \mathbb R^d$ we have
		\begin{equation}\label{eq: holder bound for nu}
		|\nu(t,x)-\nu(s,y)|  \leq C \left( |t-s|^\gamma + |x-y| \right),
		\end{equation}
		and
		\begin{equation}\label{eq: holder bound for nu_x}
		|\nabla \nu(t,x)-\nabla \nu(s,y)|  \leq C \left( |t-s|^\gamma + |x-y|^\alpha \right),
		\end{equation}
		for any  $\gamma<1-\beta-\delta$ and for $\alpha=\delta-\frac d p $.
\end{lemma}
\begin{proof}  
	 Since $u\in C([0,T]; H^{1+\delta}_p)$ and $(\delta,p)\in K(\beta,q)$, we can apply the fractional Morrey inequality (Lemma \ref{lm: fractional Morrey ineq}) and for all $t\in[0,T]$ we get $u(t)\in C^{1,\alpha}$ with $\alpha=\delta-\frac d p $. By using the definition of the norms in $C^{1,\alpha}$ and in 
	$C([0,T]; H^{1+\delta}_p(\mathbb R))$ we get \eqref{eq: sup bound for nu} for $\nu=u$. 
	
	For $\nu=u^n$, since from Lemma \ref{lm: continuity of u wrt approximation} part (i) it holds  $u^n\to u$ in $ C([0,T];H^{1+\delta}_{p})$, then there exists a constant $C$ such that 
	\begin{equation}
	\label{eq: u bound u^n}
	\|u^n\|_{\infty,H_p^{1+\delta}}\le C\|u\|_{\infty,H_p^{1+\delta}}, \quad \forall n\geq 0.
	\end{equation}
	Then we have  
	\[
	\sup_{0\leq t \leq T}\left( \sup_{x\in \R^d} |u^n(t,x)|\right) \leq \|u^n\|_{\infty,H_p^{1+\delta}}
	\le C\|u\|_{\infty,H_p^{1+\delta}}.
	\]

	For \eqref{eq: sup bound for nu_x}, we first observe that by the definition of the norm in $C^{1,\alpha}$ and the continuous embedding  $ H^{1+\delta}_p \subset C^{1,\alpha}$ we have
	\begin{align*}
		&\sup_{x\in \mathbb R} |  \nabla u(t,x)|  \leq  \| u (t)\|_{C^{1,\alpha}}
		\leq  \| u (t)\|_{H^{1+\delta}_p}\\
		\leq &\sup_{t\in[0,T]} \| u (t)\|_{H^{1+\delta}_p}
		=  \| u \|_{C([0,T];H^{1+\delta}_p)} =:C,
	\end{align*}
	where the last bound is due to the fact that $u\in C([0,T];H^{1+\delta}_p) $. This proves \eqref{eq: sup bound for nu_x} for $\nabla \nu= \nabla u$. Bound \eqref{eq: sup bound for nu_x} for $\nabla \nu=\nabla u^n$ is obtained analogously by using (\ref{eq: u bound u^n}).
	
	To prove  \eqref{eq: holder bound for nu}, let $(t,x), (s,y) \in [0,T]\times \mathbb R^d $. We have
	\begin{align*}
		&|  u(t,x) - u(s,y) |\\
		\leq&  | u(t,x) -u(s,x)| + |u(s,x) -u(s,y)|\\
		\leq & \sup_{x\in \mathbb R^d} |u(t,x) - u(s,x)| + |u(s,x) - u (s,y)|\\
		\leq &   \| u(t,\cdot) -u(s,\cdot)\|_{C^{1,\alpha}} +  \|u(s,\cdot)\|_{C^{1,\alpha}} |x-y|  \\
		\leq &   \| u(t,\cdot) -u(s,\cdot)\|_{H^{1+\delta}_p} +  \|u(s,\cdot)\|_{H^{1+\delta}_p} |x-y|  \\
		\leq & \|u\|_{C^{0,\gamma}([0,T]; H^{1+\delta}_p)} |t-s|^\gamma +  \|u\|_{C^{0,\gamma}([0,T]; H^{1+\delta}_p)} |x-y|  \\
		\leq & \|u\|_{C^{0,\gamma}([0,T]; H^{1+\delta}_p)} (|t-s|^\gamma +|x-y|),
	\end{align*}
	having used the embedding property (fractional Morrey inequality) with $\alpha= \delta - d/p$, the Lipschitz property of $u(t, \cdot)$ (due to the fact that it is differentiable) and the H\"older property of $u(\cdot)$ with values in $H^{1+\delta}_p$. Setting $C= \|u\|_{C^{0,\gamma}([0,T]; H^{1+\delta}_p)}  $ concludes the proof of  \eqref{eq: holder bound for nu} for $\nu=u$. 
	
	The bound \eqref{eq: holder bound for nu} for $\nu=u^n$ is obtained from the previous one: we proceed as the proof for $\nu=u$ and get
	\begin{equation}\label{eq: bound u^n dep on n}
	|u^n(t,x)-u^n(s,y)|  \leq \|u^n\|_{C^{0,\gamma}([0,T]; H^{1+\delta}_p)} \left( |t-s|^\gamma + |x-y| \right).
	\end{equation}
    Plugging (\ref{ineq: Cgamma norm bound}) from Lemma \ref{lm: uniform bound u^n} into \eqref{eq: bound u^n dep on n}, we get the desired result.

	To show \eqref{eq: holder bound for nu_x} for $\nabla \nu = \nabla u$ we proceed with very similar computations for $| \nabla u(t,x) - \nabla u(s,y) | $ as in the proof of \eqref{eq: holder bound for nu}, but now we use the fact that $\nabla u(s, \cdot )$ is only H\"older continuous of order $\alpha$ rather than Lipschitz continuous, that is $|\nabla u(s,x) - \nabla u (s,y)| \leq \|u(s,\cdot)\|_{C^{1,\alpha}} |x-y|^\alpha$, so we finally have 
	\begin{equation*}
		\begin{aligned}
	&| \nabla  u(t,x) - \nabla u(s,y) | \\
	\leq &|\nabla u(t,x)-\nabla u(s,x)|+|\nabla u(s,x)-\nabla u(s,y)|\\
	\le& \|u(t,\cdot)-u(s,\cdot)\|_{C^{1,\alpha}}+\|u(s,\cdot)\|_{C^{1,\alpha}}|x-y|^\alpha\\
	\leq &\|u\|_{C^{0,\gamma}([0,T]; H^{1+\delta}_p)} (|t-s|^\gamma +|x-y|^\alpha),
	    \end{aligned}
	\end{equation*}
	which is the claim with $C$ as in the previous bound.
	
	The proof of \eqref{eq: holder bound for nu_x} for $\nabla \nu= \nabla u^n$ is similar and uses (\ref{ineq: Cgamma norm bound}) in the last part.
\end{proof}

\section{Solution of BSDE \eqref{eq: forward-backward SDE transformed}}\label{sc: forward-backward SDE transformed}

\subsection{Definition of solution, existence and uniqueness}
In this section we  consider FBSDE \eqref{eq: forward-backward SDE transformed}, which we write again below for convenience
\begin{equation}\label{eq: FBSDE transformed}
	\left\{
	\begin{array}{l}	
		{X}_s^{t,x} = \ x  +\int_t^s \mathrm dW_r,\\
		{Y}_s^{t,x} = \  \Phi({X}_T^{t,x})  - \int_s^T{Z}_r^{t,x} \mathrm d{W}_r + \int_s^T f(r,{X}_r^{t,x},{Y}_r^{t,x},{Z}_r^{t,x}) \mathrm dr\\ 	
		\qquad +  \int^T_s {{Z}_r^{t,x} b(r,{X}_r^{t,x})}\mathrm{d} r,\\
		{\forall }  s\in[t,T],
	\end{array}\right.
\end{equation}
where $(W_s)_s$ is a given Brownian motion on a filtered probability space $(\Omega, \mathcal F,   \PP, \mathbb F)$ and the filtration $\mathbb F$   is the Brownian filtration. Here $f:[0,T]\times \mathbb R^d\times \mathbb R^m\times \mathbb R^{m\times d } \to \mathbb R^m$ and $\Phi: \mathbb R^d\to \mathbb R^m$. We note that $X^{t,x}:=(X_s^{t,x})_{s\in[t,T]}$ is in fact a Brownian motion starting from $x$ at time $t$.  
The major difficulty related to \eqref{eq: FBSDE transformed} is the term $\int^T_s {  {Z}_r^{t,x}b(r, {X}_r^{t,x})}\mathrm{d} r$ because $b\in L^\infty([0,T]; H^{-\beta}_{ q})$.  Given $ X^{t,x}$, we introduce the notion of \emph{virtual-strong solution} for the backward SDE in (\ref{eq: FBSDE transformed}). To do so,  we first consider the following auxiliary PDE
\begin{equation}\label{eq: PDE aux}
	\left\{
	\begin{array}{ l}
		w_t+\frac12 \Delta w = { \nabla u\, b},\\
		w(T,x) = 0,  \\
		{\forall }  (t,x)\in[0,T]\times \mathbb R^d,
	\end{array}
	\right.
\end{equation}
where $u$ is the mild solution of (\ref{eq: PDE backward Rd}). The term $ \nabla u \, b$ is defined by means of the pointwise product, and thanks to the semigroup properties (see Section \ref{sc: preliminaries} for more details) there exists a unique mild solution $w\in C([0,T]; H^{1+\delta}_p) $ to \eqref{eq: PDE aux} which is given by 
\begin{align}\label{eq: mild solution w} \nonumber
	w(t)&= {P_p({T-t})w(T)} +\int^T_t P_p(r-t){\nabla u(r)b(r)}\d r\\
	&= \int^T_t P_p(r-t){\nabla u(r) b(r)}\d r.
\end{align}
Note that by the Fractional Morrey inequality (Lemma \ref{lm: fractional Morrey ineq}) we have that $w$ can be evaluated  pointwisely since $w\in C([0,T]; C^{1,\alpha}) $ for $\alpha = \delta-\frac d p$. We use this function $w$ to give a meaning to the backward SDE in  (\ref{eq: FBSDE transformed}) as follows.  In the sequel we will drop the superscript $t,x$ for simplicity of notation.
\begin{definition}\label{def: virtual solution BSDE}
	A {\em virtual-strong solution} to  the backward SDE in  \eqref{eq: FBSDE transformed} is a couple $( {Y},  {Z})$ such that
	\begin{itemize}
		\item $Y$ is continuous and $\mathbb F$-adapted and $Z$ is $\mathbb F$-progressively measurable; 
		\item  $ \EE \left[ \sup_{r\in[t,T]}|Y_r|^2 \right ] < \infty$ and $ \EE\left[ \int_t^T |Z_r|^2 \mathrm dr\right] < \infty$;
		\item for all $s\in[t,T]$, the couple satisfies the following  backward SDE 
		\begin{align}
			\label{eq: BSDE virtual}
			{Y}_s= & \ \Phi( {X}_T)-\int^T_s {Z}_r\d  {W}_r+\int^T_sf(r, {X}_r, {Y}_r,{Z}_r)\d r\nonumber\\
			&-w(s, {X}_s)-\int^T_s \nabla w(r,  {X}_r)\d  {W}_r
		\end{align}
		$\PP$-almost surely, where $w$ is the solution of \eqref{eq: PDE aux} given by \eqref{eq: mild solution w}.
	\end{itemize}
\end{definition}
An intuitive  explanation on why we define virtual-strong solutions like this is the fact that if $b$ were  smooth, also $w$ would  be  smooth and we could apply It\^{o}'s formula to $w(\cdot,  {X})$, where $X_s=x+W_s-W_t$, to get
\begin{equation*}
	\begin{aligned}
		\d w(s,  {X}_s)&=w_t(s,  {X}_s)\d s+ \nabla w(s,  {X}_s)\d   {X}_s +\frac12 \Delta w(s,  {X}_s)\d s\\
		&={ \nabla u(s,  {X}_s) b(s,  {X}_s)}\d s + \nabla w(s,  {X}_s)\d  {W}_s.
	\end{aligned}
\end{equation*}
Therefore, we could write 
\begin{align*}
	&{w(T, {X}_T)} -w(s, {X}_s)-\int^T_s \nabla w(r,  {X}_r)\d  {W}_r\\
	&=  -w(s, {X}_s)-\int^T_s \nabla w(r,  {X}_r)\d  {W}_r\\
	&= 	\int^T_s  \nabla u(r,  {X}_r)b(r,  {X}_r)\d r \\
	&= 	\int^T_s {  Z_r b(r,  {X}_r)}\d r ,
\end{align*}
where the last equality holds because in the smooth case the solution $(Y,Z)$ could be written as $( u(\cdot,  {X}),\nabla u(\cdot,  {X}) )$. This is why the term $-w(s, {X}_s)-\int^T_s \nabla w(r,  {X}_r)\d  {W}_r$ appears in \eqref{eq: BSDE virtual} in place of $\int^T_s{  Z_r b(r,  {X}_r)}\d r$.

We recall that a {\em strong solution} of \eqref{eq: BSDE virtual} is a couple $(Y , Z )$ such that \begin{itemize}
	\item $Y  $ is continuous and $\mathbb F$-adapted, $Z $ is $\mathbb F$-progressively measurable;
	\item  $ \EE \left[ \sup_{r\in[t,T]}|Y_r |^2 \right ] < \infty$ and $ \EE\left[ \int_t^T |Z_r |^2 \mathrm dr\right] < \infty$;   
	\item \eqref{eq: BSDE virtual} holds $\PP$-almost surely.
\end{itemize}
Note that the terms involving $w$ in \eqref{eq: BSDE virtual} do not pose any extra condition because we can prove that $w$ is continuous and bounded (see Lemma \ref{lm: properties of w and nabla w} below).

The notion of virtual-strong solution for BSDE is in alignment with classical strong solutions when the drift $b$ is a function with classical regularity properties. In this case   a virtual-strong solution is also a strong solution, as illustrated  in the proposition below.
\begin{proposition}
	Let $b\in C([0,T]; C^1_b(\mathbb R^d, \mathbb R^d))$ (bounded with bounded first derivatives).
	Then the virtual-strong solution $(Y,Z)$  of the backward SDE in \eqref{eq: FBSDE transformed}   is also a strong solution.
\end{proposition}
\begin{proof}
	First observe that the first two conditions for $Y$ and $Z$ in Definition \ref{def: virtual solution BSDE} are the same as for strong solutions.
	
	Let $u$ be the classical solution of \eqref{eq: PDE backward Rd} and $w$  be the classical solution of \eqref{eq: PDE aux}. Then $u$ and $w$ are both at least of class $C^{1,2}$ and  by It\^o's formula applied to $w$ we have that the term $ -w(s, {X}_s)-\int^T_s \nabla w(r,  {X}_r)\d  {W}_r$ is equal to $\int^T_s { {Z}_r b(r, {X}_r)}\mathrm{d} r$, hence  the BSDE in \eqref{eq: FBSDE transformed} holds $\PP$-a.s.. 
\end{proof}

We remark that, although every term in the backward SDE \eqref{eq: BSDE virtual} is well defined, this SDE is not written in a classical form. Hence to find a virtual-strong solution we transform \eqref{eq: BSDE virtual} using the solution of the PDE \eqref{eq: PDE aux}, in particular we apply the transformation $y\mapsto y+w(s,x)$ where $w$ is the solution of the PDE (\ref{eq: PDE aux}). This transformation could be regarded as the analogous of the Zvonkin transformation for SDEs to get rid of a (singular) drift. More precisely, we set $\wh Y_s := Y_s+w(s,X_s) $ and $\wh Z_s := Z_s+\nabla w(s,X_s) $ for all $s\in[t,T]$ and $\wh{f}(r,x,y,z):=f(r,x,y-w(r,x), z-\nabla w(r,x))$, and  we get the following {\em auxiliary backward    SDE}
\begin{align}
	\label{eq: BSDE virtual equivalent}
	\wh{Y}_s=\Phi( {X}_T)-\int^T_s\wh{Z}_r\d  {W}_r+\int^T_s\wh{f}(r, {X}_r,\wh{Y}_r,\wh{Z}_r)\d r,
\end{align}
for all $s\in[t,T]$.

It turns out that indeed the  BSDEs  (\ref{eq: BSDE virtual}) and (\ref{eq: BSDE virtual equivalent}) are equivalent as shown in the following proposition.
\begin{proposition}\label{prop: equivalence of backward SDEs}
	Let $ {X}$ be a  Brownian motion starting from $x$ at time $t$ and $\mathbb F$ be the Brownian filtration generated by $W$. Then 
	\begin{itemize}
		\item[(i)] If $( {Y}, {Z})$ is a virtual-strong solution of the backward SDE in \eqref{eq: FBSDE transformed}, then 
		$$(\wh{Y},\wh{Z}):=( {Y}+w(\cdot, {X}), {Z}+\nabla w(\cdot, {X}))$$ is a strong solution of  \eqref{eq: BSDE virtual equivalent}.
		\item[(ii)] If $(\wh{Y},\wh{Z})$ is a strong solution of \eqref{eq: BSDE virtual equivalent}, then $$( {Y} , {Z} ):=(\wh{Y}-w(\cdot, {X}),\wh{Z}-\nabla w(\cdot, {X}))$$ is a virtual-strong solution of the backward SDE in  \eqref{eq: FBSDE transformed}.
	\end{itemize}
\end{proposition}

\begin{proof}
The proof is very easy and straight-forward, so we omit it.
\end{proof}

We will now prove existence and uniqueness of the virtual-strong solution for the FBSDE \eqref{eq: FBSDE transformed}. For this we need Assumption \ref{ass: f in Rd}.

\begin{theorem}\label{theor: exist uniq of backward virtual solution}
	Under Assumption \ref{ass: f in Rd} there exists a unique virtual-strong solution $(Y,Z)$ to the backward SDE in (\ref{eq: FBSDE transformed}). 
\end{theorem}
\begin{proof}
	By definition, a virtual-strong solution of the backward SDE in (\ref{eq: FBSDE transformed}) is a couple that solves   BSDE  \eqref{eq: BSDE virtual}, if $u$ exists. Note that by Remark \ref{rm: link between Assumptions} we know that Assumption \ref{ass: f in Rd} implies Assumption \ref{ass: f in sobolev space}, hence $u $ does exist by Theorem \ref{thm: PDE backward}. Moreover 	BSDE  \eqref{eq: BSDE virtual}  is equivalent to BSDE   \eqref{eq: BSDE virtual equivalent} by Proposition \ref{prop: equivalence of backward SDEs}.
	 
	Using the Lipschitz assumption on $f$ from Assumption \ref{ass: f in Rd} and the definition of $\wh f$, we have for any {$y,y' \in \mathbb R^m$ and $ z, z' \in \mathbb R^{m\times d}$} that 
	\begin{equation*}
		\begin{aligned}
			&|\wh{f}(t,x, y, z)-\wh{f}(t, x ,y^{\pr}, z' )|\\
			=&|f(t,x ,y-w(t,x ), z-\nabla w(t, x) )-f(t,x ,y^{\pr}-w(t, x ), z'-  \nabla w(t, x) )|\\
			\le&C( |y-y^{\pr}|+|z-z'|).
		\end{aligned}
	\end{equation*}
	Moreover  by definition of $\wh f$ we have  
	\begin{equation*}
		\begin{aligned}
			&	\EE \left[ \int^T_0  |\wh{f}(r,x+ W_r, 0,0)|^2\d r \right]\\  
			= &  \EE \left[ \int^T_0 |f(r,x+ W_r,  -w(r,x+ W_r),  -\nabla w(r, x+ W_r) )|^2\d r \right]\\
			\le &  \ C \left(1+ E \left[ \int^T_0  |f(r,x+ W_r, 0,0)|^2  \d r\right]\right) ,
		\end{aligned}
	\end{equation*} 
where we have used the fact that  $w$ and $\nabla w$ are uniformly bounded by Lemma \ref{lm: properties of w and nabla w}. The latter integral is bounded using the  assumption of $f(t,x,0,0)$, indeed
\begin{align*}
\EE \left[ \int^T_0  |f(r,x+ W_r, 0,0)|^2  \d r \right]&  \leq  \EE \int^T_0 \sup_{t,x} |f(t,x, 0,0) |^2  \d r  \leq c. 
\end{align*}
Hence   equation  \eqref{eq: BSDE virtual equivalent}   has a unique strong solution by classical results (see for example \cite[Theorem 2.1]{ElKaroui-et.al.97}).
\end{proof}

\subsection{The auxiliary PDE and the auxiliary BSDE} \label{sc: auxiliary PDE and BSDE}
We now establish several useful properties for the auxiliary PDE \eqref{eq: PDE aux} and for the auxiliary BSDE \eqref{eq: BSDE virtual equivalent}, which will be used in the next Section to prove the non-linear Feynman-Kac formula.

We start by proving a result analogous to Lemma \ref{lm: properties of solution u and u^n of PDE}.
\begin{lemma}\label{lm: properties of w and nabla w} 
	Let Assumption \ref{ass: f in sobolev space} hold and $b\in L^{\infty}(0,T; H^{-\beta}_{q})$. Then the solution $w$ is an element of $C^{0,\gamma}([0,T];H^{1+\delta}_p)$  for all $2\gamma<1-\delta-\beta$  and it enjoys the following bounds
	\begin{align}\label{eq: sup bound w}
		&\sup_{0\leq t \leq T}\left( \sup_{x\in \mathbb R^d} |w(t,x)|\right) \leq C,\\\label{eq: sup bound nabla w}
		&\sup_{0\leq t \leq T}\left( \sup_{x\in \mathbb R^d} | \nabla w (t,x)| \right) \leq C.
	\end{align}
	Furthermore, for all $t,s\in[0,T]$ and $x, y \in \mathbb R^d$ we have
	\begin{align}\label{eq: holder bound w}
		&|w(t,x)-w(s,y)|  \leq C \left( |t-s|^\gamma + |x-y| \right),\\
		&|\nabla w(t,x)-\nabla w(s,y)|  \leq C \left( |t-s|^\gamma + |x-y|^\alpha \right), \label{eq: holder bound nabla w}
	\end{align}
	where $\alpha=\delta-\frac{d}{p}$.
\end{lemma}

\begin{proof}
	To show that $w\in C([0,T]; H^{1+\delta}_p) $ we first  observe that 
	$ { \nabla u\,b}\in L^\infty([0,T];H^{-\beta}_p)$ since
	$$
	\| {\nabla u(s)b(s)}\|_{H^{-\beta}_p}\le C   \|\nabla u(s)\|_{H_p^\delta}\|b(s)\|_{H^{-\beta}_q},
	$$
	and taking the supremum over $s\in[0,T]$ the right-hand side is bounded by a constant which is independent of $s$. Hence 
	$$
	\| {\nabla u\,b}\|_{\infty, H^{-\beta}_p} \leq \sup_{0\le s\le T} C \|u(s)\|_{H^{1+\delta}_p}\|b\|_{\infty,H^{-\beta}_q}\le C(b,u).
	$$
	By  Proposition \ref{prop: 11inFlIsRu} applied to equation \eqref{eq: mild solution w} we have that $w\in C^{0,\gamma}([0,T];H^{2-2\ep-\beta}_p) $ for every $\ep>0$ and $\gamma \in(0,\ep)$, and setting with $\varepsilon =  \frac{1-\delta-\beta}{2} $  it implies $ w\in C([0,T]; H^{1+\delta}_p)$.
	
	The bounds \eqref{eq: sup bound w} and \eqref{eq: sup bound nabla w} follow by fractional Morrey inequality (Lemma \ref{lm: fractional Morrey ineq})
	$$
	w\in C^{0,\gamma}([0,T];H^{1+\delta}_p)\subset C^{0,\gamma}([0,T];C^{1,\alpha}),
	$$
	where $\alpha=\delta-\frac{d}{p}$. Hence the $\sup_{t,x}$ of the functions $w$ and $\nabla w$ are finite.
	
	The bound \eqref{eq: holder bound w} is clear by using the norm definition in $C^{0,\gamma}$, whereas \eqref{eq: holder bound nabla w} can be obtained by using the fact that  $  w\in C^{0,\gamma}([0,T]; C^{0,1+\alpha})$ implies  $\nabla w\in C^{0,\gamma}([0,T]; C^{0,\alpha})$  and applying the definition of the  norm in the latter space.
\end{proof}

If we  now consider a smooth coefficient $b^n$ in place of $b$  then the PDE 
(\ref{eq: PDE aux})  becomes
\begin{equation}
	\label{eq: PDE aux n}
	\left\{
	\begin{array}{l}
		w^n_t+\frac12 \Delta w^n={ \nabla u^n\,b^n},\\
		w^n(T,x)=0, \\
		{\forall }  (t,x)\in[0,T]\times \mathbb R^d.
	\end{array}
	\right.
\end{equation}

For this approximating PDE we have nice convergence properties as follows. 

\begin{lemma}\label{lm: properties of wn}
	Let Assumption \ref{ass: f in sobolev space} hold and let $b^n\to b$ in $L^{\infty}([0,T];H^{-\beta}_{q})$. Then $w^n\to w$ in $C^{0,\gamma}([0,T]; C^{1,\alpha})$ and $\nabla w^n\to \nabla w$ in $C^{0,\gamma}([0,T]; C^{0,\alpha})$. 
	In particular,  $w^n(t,x)\to w(t,x)$ and $\nabla w^n(t,x)\to \nabla w(t,x)$ uniformly on $[0,T]\times\R^d$.
\end{lemma}
\begin{proof}
	By Lemma \ref{lm: properties of w and nabla w} we have that $w$ and $w^n$ are both elements of $C^{0, \gamma}([0,T]; H^{1+\delta}_p)  $. 
	The norm of $w-w^n$ in $C^{0, \gamma}([0,T]; H^{1+\delta}_p)$ has  two terms, as recalled in Section \ref{sc: preliminaries}. The  first one can be bounded by observing that  
	\begin{align*}
		w(T-t)-w^n(T-t ) 
		&=  \int_0^t P_p(r) \big({  \nabla u (r+T-t)b(r+T-t)}\\
		& \quad -   { \nabla u^n(r+T-t) b^n(r+T-t)}\big) \d r 
	\end{align*}
	and by abuse of notation we consider the semigroup simply  acting on ${  \nabla u (r)b(r)-   \nabla u^n(r)b^n(r)}$ because the regularity properties are the same. So
	\begin{align*}
		& \| w(T-t)-w^n(T-t)  \|_{H^{1+\delta}_p}\\
		&\leq  \left  \|  \int_0^t P_p(r) ({  \nabla u (r)b(r)-   \nabla u (r)b^n(r)} ) \d r \right\|_{H^{1+\delta}_p}\\
		&\quad + \left  \| \int_0^t P_p(r) ({  \nabla u (r)b^n(r)-   \nabla u^n (r)b^n(r)} ) \d r \right \|_{H^{1+\delta}_p} \\
		&\leq  \int_0^t r^{-\frac{1-\delta-\beta}{2}} \Big (  \|  u (r) \|_{H^{1+\delta}_p}     \|  b(r)-   b^n (r)   \|_{H^{-\beta}_q} \\
		&\quad +    \|   b^n(r) \|_{H^{-\beta}_q}   \|  \nabla u (r)-   \nabla u^n (r)   \|_{H^{1+\delta}_p} \Big) \d r \\
		&\leq C T^{\frac{1+\delta+\beta}{2}}  \|  b -   b^n     \|_{\infty, H^{-\beta}_q} ,
	\end{align*}
	where the constant $C$ is independent of $n$ (for $n$ large enough) because $u^n\to u$ as shown in Lemma \ref{lm: continuity of u wrt approximation}, part (i) and $b_n\to b$ by hypotheses. Thus 
\begin{align*}	
	\sup_{t\in[0,T]} \| w(t)-w^n(t)  \|_{H^{1+\delta}_p}&= \sup_{t\in[0,T]} \| w(T-t)-w^n(T-t)  \|_{H^{1+\delta}_p}\\
	&\leq  C    \|  b -   b^n     \|_{\infty, H^{-\beta}_q}.
\end{align*}
	The H\"older term in the norm of $w-w^n$ can be bounded by using Proposition \ref{prop: 11inFlIsRu} with $\ep = \frac{1-\delta-\beta}{2} $,  since the integrand  $h(r):= b(r) \nabla u (r)-  b^n(r) \nabla u^n (r)$ belongs to $H^{-\beta}_p$. Then  we have 
	\begin{equation*} 
		\frac{ \|w^n(t)-w(t) -(w^n(s)-w(s))\|_{H^{1+\delta}} }{|t-s|^\gamma}   \le 	 C\|h\|_{\infty, H^{-\beta}_p},
	\end{equation*}	
	where $C$ is independent of $n$ 
	and the norm of $h$ is   bounded by $C    \|  b -   b^n     \|_{\infty, H^{-\beta}_q} $  as done above. Hence we have shown that 	$$ w^n \to w \text{ in }  C^{0, \gamma}([0,T]; H^{1+\delta}_p)$$ 
	which implies
	$$ \nabla w^n \to \nabla w \text{ in }    C^{0, \gamma}([0,T]; H^{ \delta}_p)$$ by the continuity of the mapping $\nabla: H^{1+\delta}_p \to H^\delta_p$. 
	
	By the Sobolev embedding (Lemma \ref{lm: fractional Morrey ineq}) we have $C^{0, \gamma}([0,T]; H^{1+\delta}_p)\subset {C^{0,\gamma}([0,T]; C^{1,\alpha})}$ and so it follows that 
	\begin{equation*}
		\begin{aligned}
			\sup_{0\le t\le T}\sup_{x\in\R^d} |w^n(t,x)-w(t,x)|\le C \|b-b^n\|_{ \infty ,H^{-\beta}_q }
		\end{aligned}
	\end{equation*}	 
	and 
	\begin{equation*}
		\begin{aligned}
			\sup_{0\le t\le T}\sup_{x\in\R^d} |\nabla w^n(t,x)-\nabla w(t,x)|\le C \|b-b^n\|_{\infty , H^{-\beta}_q},
		\end{aligned}
	\end{equation*}
	which is the uniform convergence claimed.
\end{proof}

\subsection{Feynman-Kac representation formula}
In this last section we will establish a non-linear Feynman-Kac representation formula for the FBSDE \eqref{eq: FBSDE transformed} using the solution of the PDE \eqref{eq: PDE backward Rd} and of the auxiliary PDE \eqref{eq: PDE aux}.  In particular, we will construct the virtual-strong solution of \eqref{eq: FBSDE transformed} --that is a strong solution of \eqref{eq: BSDE virtual}-- by means of the mild solution of the PDE \eqref{eq: PDE backward Rd}, and we will also show that the unique mild solution can be obtained as the first component $Y$ at initial time $t$ of the virtual-strong solution $(Y,Z)$, and in this case the gradient of the solution corresponds to $Z$.

{
\begin{theorem}\label{theor: Feynman-Kac construct virtual solution}
	Let Assumption \ref{ass: f in Rd} hold.	Let $u$ be the unique mild solution of (\ref{eq: PDE backward Rd}) and $X$ be the solution of the forward equation in (\ref{eq: FBSDE transformed}), namely $X_s = x+ W_s-W_t$, $s\in[t,T]$. Then the couple $(u(\cdot, {X}),\nabla u(\cdot, {X}))$ is a virtual-strong solution of the backward SDE in (\ref{eq: FBSDE transformed}). 
\end{theorem}
\begin{proof}
  First we note that by Remark \ref{rm: link between Assumptions} we can consider the composition of  $f$ with $u, \nabla u$ and this satisfies Assumption \ref{ass: f in sobolev space}.  Hence by Theorem \ref{thm: PDE backward} we know that a  solution $u$ to PDE \eqref{eq: PDE backward Rd} exists and it is unique. Furthermore this solution is in $C([0,T];C^{1,\alpha})$ for some small $\alpha>0$ by Lemma \ref{lm: properties of solution u and u^n of PDE} and it is uniformly bounded in $(t,x)$. These properties, together with the fact that $X$ is a Brownian motion starting in $x$ at time $t$, imply that the first two bullet points of Definition  \ref{def: virtual solution BSDE} are easily satisfied for the couple   $(u(\cdot, {X}),\nabla u(\cdot, {X}))$. The only non-trivial point to verify in this definition is to show that  $(u(\cdot, {X}),\nabla u(\cdot, {X}))$ satisfies \eqref{eq: BSDE virtual}, where $w$ is given by \eqref{eq: PDE aux}.
  
 To show this  we take a  smooth approximating sequence,  e.g.\ $b^n \in L^\infty(0,T; C^1_b(\R^d;\R^d))$, such that  $b^n\to b$  converges in $L^\infty(0,T; L^{-\beta}_{q})$. The PDE (\ref{eq: PDE backward Rd}) then  becomes
	\eqref{eq: PDE backward for u^n} and PDE (\ref{eq: PDE aux})  becomes	\eqref{eq: PDE aux n}.
	These approximations are smooth so we can apply It\^o's formula to both $u^n(\cdot, {X})$ and $w^n(\cdot, {X} )$, and get
	\begin{equation*}
		\begin{aligned}
			\d u^n(s, {X}_s)=&-{ \nabla u^n(s, {X}_s) b^n(s, {X}_s)}\d s\\
			&-f(s, {X}_s,u^n(s, {X}_s),\nabla u^n(s, {X}_s))\d s+\nabla u^n(s, {X}_s)\d  {W}_s	,
		\end{aligned}
	\end{equation*}
	and
	\begin{equation*}
		\d w^n(s, {X}_s)={  \nabla u^n(s, {X}_s)b^n(s, {X}_s)}\d s+\nabla w^n(s, {X}_s)\d  {W}_s.	
	\end{equation*}
	Adding the second equation to the first we get rid of the term with ${\nabla u^nb^n}$ and we end up with
	\begin{equation*}
		\begin{aligned}
			\d u^n(s, {X}_s)= & -\d w^n(s, {X}_s)-f(s, {X}_s,u^n(s, {X}_s),\nabla u^n(s, {X}_s))\d s\\
			&+\nabla w^n(s, {X}_s)\d  {W}_s+\nabla u^n(s, {X}_s)\d  {W}_s	.
		\end{aligned}
	\end{equation*}
	Integrating from $s$ to $T$  gives
	\begin{equation}\label{eq: Feynman-Kac approx}
		\begin{aligned} 
			u^n(s, {X}_s)=\ &\Phi( {X}_T)- w^n(s, {X}_s)\\
			&+\int^T_sf(r, {X}_r,u^n(r, {X}_r),\nabla u^n(r, {X}_r))\d r\\
			&-\int^T_s\nabla w^n(r, {X}_r)\d  {W}_r-\int^T_s\nabla u^n(r, {X}_r)\d  {W}_r.	
		\end{aligned}
	\end{equation}	
	
Our aim to show that the limit of \eqref{eq: Feynman-Kac approx} is given by 
	\begin{equation}\label{eq: Feynman-Kac}
		\begin{aligned} 
			u(s, {X}_s)=\ &\Phi( {X}_T)- w(s, {X}_s)\\
			&+\int^T_sf(r, {X}_r,u(r, {X}_r),\nabla u(r, {X}_r))\d r\\
			&-\int^T_s\nabla w(r, {X}_r)\d  {W}_r-\int^T_s\nabla u(r, {X}_r)\d  {W}_r.	
		\end{aligned}
	\end{equation}	
We will consider the limit in $\mathbb S^2$: For a stochastic process $(\xi_s)_{t\leq s \leq T}$ the norm in $\mathbb S^2$ is given by $\mathbb E [\sup_{t\leq s\leq T} |\xi_s|^2]$. We take the difference of \eqref{eq: Feynman-Kac approx} and 	\eqref{eq: Feynman-Kac}, then by triangular inequality is enough to show $\mathbb S^2$-convergence to zero for each of the following five  terms:
\begin{align*}
&u^n(\cdot, X) - u(\cdot, X)\\
&w^n(\cdot, X) - w(\cdot, X)\\
&\int_\cdot^T  f(r, X_r, u^n(r, X_r), \nabla u^n(r, X_r)  \mathrm dr  \\
&\phantom{some space} - \int_\cdot^T f(r, X_r, u(r, X_r), \nabla u(r, X_r) ) \mathrm dr \nonumber \\
&\int_\cdot^T  \nabla u^n(r, X_r)  \mathrm dW_r - \int_\cdot^T \nabla u(r, X_r)   \mathrm dW_r \\
&\int_\cdot^T  \nabla w^n(r, X_r)  \mathrm dW_r  - \int_\cdot^T \nabla w(r, X_r)   \mathrm dW_r .
\end{align*}
The first two are a consequence of uniform convergence of $u^n$ to $u$ and $w^n$ to $w$ (which is proven in Lemma \ref{lm: continuity of u wrt approximation} and \ref{lm: properties of wn}). The third term converges to zero thanks to the Lipschitz continuity of $f$ (by Assumption \ref{ass: f in Rd}) and  uniform convergence of $u^n$ and $\nabla u^n$ (again by Lemma \ref{lm: continuity of u wrt approximation}). The last two terms can be bounded using BDG inequality and Lemma  \ref{lm: properties of wn})  as follows (we show it only for $w$, the same applies to  $u$ thanks to Lemma  \ref{lm: continuity of u wrt approximation}.
\begin{align*}
\mathbb E &\left [\sup_{t\leq s\leq T} \left|\int_s^T  (\nabla w^n(r, X_r) - \nabla w(r, X_r)    )   \mathrm dW_r   \right|^2\right ]\\
&\leq c \mathbb E \left [\int_s^T  (\nabla w^n(r, X_r) - \nabla w(r, X_r)    )^2   \mathrm dr  \right ]\\
&\leq c \mathbb E \left [ \int_s^T  \left(  \sup_{r,x}|\nabla w^n(r, X_r) - \nabla w(r, X_r) | \right  )^2   \mathrm dr   \right ] \to 0.
\end{align*}
This concludes the proof. 
\end{proof}
}

From  Theorem \ref{theor: Feynman-Kac construct virtual solution} and using Proposition \ref{prop: equivalence of backward SDEs}, it is also easily seen that $(u(\cdot, {X})+w(\cdot, {X})),\nabla u(\cdot, {X}) +\nabla w(\cdot, {X}))$ is a strong solution of (\ref{eq: BSDE virtual equivalent}), where $u $ is the solution of PDE \eqref{eq: PDE backward Rd} and $w$ is the solution of \eqref{eq: PDE aux}.

Next we have the opposite result, namely that the BSDE provides a representation for the mild solution of the PDE. For this result we resume the use of the superscript $t,x$ for better clarity.

\begin{theorem}\label{theor: Feynman-Kac construct mild solution}
	Let Assumption \ref{ass: f in sobolev space} hold, and let $( {Y}^{t,x}, {Z}^{t,x})$ be a virtual-strong solution of the backward SDE in  (\ref{eq: FBSDE transformed}). Assume further that there exists deterministic functions $\alpha(\cdot, \cdot) $ and $\beta(\cdot, \cdot) $ such that 
	\[
	Y_s^{t,x} =  \alpha(s, X_s^{t,x}) \quad \text{and} \quad Z_s^{t,x} =  \beta(s, X_s^{t,x})
	\]
	for all $s\in[0,T]$. Moreover assume that $\alpha\in C^\varepsilon([0,T]; H^{1+\delta}_p)$ (form some $\varepsilon>0$) and $\beta\in C([0,T]; H^{\delta}_p)$. Then the unique mild solution of (\ref{eq: PDE backward Rd}) can be written as $u(t,x)=  {Y}_t^{t,x}$. Moreover we have that $\nabla u(t,x)=   {Z}_t^{t,x}$.
\end{theorem}

\begin{proof}
	Since $( {Y}^{t,x}, {Z}^{t,x})= (\alpha(\cdot, X^{t,x}), \beta(\cdot, X^{t,x}))$ is a virtual-strong solution of the backward SDE in  \eqref{eq: FBSDE transformed}, we have  for $s=t$
	\begin{align}\label{eq: Y for mild solution}
		\alpha(t,x)=\ &\Phi ( {X}_T^{t,x})-\int^T_t \beta(r, X_r^{t,x}) \mathrm d  {W}_r\\\nonumber
		&+\int^T_t f(r, {X}_r^{t,x}, \alpha(r, X_r^{t,x}),\beta(r, X_r^{t,x}))\mathrm d r\\\nonumber
		&-w(t,x)-\int^T_t\nabla w(r, {X}_r^{t,x})\mathrm d  {W}_r.
	\end{align}
	Note that  the stochastic integrals in \eqref{eq: Y for mild solution} have zero-mean because both integrands are square integrable.
	We denote by $\PP_{t,x}$	the probability measure  of $X^{t,x}$ (which we recall is a Brownian motion starting in $x$  at $t$) 	and by $\EE_{t,x}$ the expectation under this measure, namely $\EE[X^{t,x}_s] = \EE_{t,x}[X_s]$, where $X_s$ is the canonical process. Moreover,  this process $X$ generates the heat semigroup under this measures, namely for all bounded and measurable $a$ we have   
	$$\EE\left [a(s,X^{t,x}_s)\right] = \EE_{t,x}\left[a(s,X_s)\right] = \left(P(s-t) a(s,\cdot)\right) (x).$$ 
	The heat semigroup $P$ coincides with the semigroup $P_p$ when it acts on elements in $L^p$.  
	Then   taking the expectation $\EE$ on both sides of \eqref{eq: Y for mild solution}  we get
	\begin{align} \nonumber
		\alpha(t,\cdot)= &{{\EE}}\left[\Phi ( {X}_T^{t,\cdot})\right ]-w(t,\cdot)\\ \nonumber
		&+ { {\EE}} \left[\int^T_tf(r, {X}_r^{t,\cdot},  \alpha(r, X_r^{t,\cdot}),\beta(r, X_r^{t,\cdot}) )\mathrm d r \right]\\ \nonumber
		=& P_p({T-t})\Phi -w(t )+\int^T_t P_p({r-t})f(r,\cdot,\alpha(r),\beta(r))\mathrm d r\\  \label{eq: mild form for alpha and beta}
		=& P_p({T-t})\Phi + \int^T_tP_p({r-t})\left({ \nabla u(r)\, b(r)}\right ) \mathrm d r \\ \nonumber
		&+\int^T_t P_p({r-t})f(r,\cdot,\alpha(r),\beta(r))\mathrm d r,
	\end{align}
	having used in the last equality that   $w$ is the mild solution of (\ref{eq: PDE aux}).
	Next we calculate the  covariation of $Y$ and $W$. We use the covariation defined in \cite{gozzi_russo06}, recalled below for convenience:
\[
\left[  Y, W \right]_s	:=\lim_{\varepsilon \to 0 } \frac{1}{\varepsilon} \int_0^s (Y_{r+\varepsilon} -Y_r)(W_{r+\varepsilon} -W_r) \mathrm dr,
\]
if the limit exists \emph{u.c.p.}\ in $s$.	
Notice that $\alpha \in C^\varepsilon([0,T]; H_p^{1+\delta})$ implies by fractional Morrey inequality (Lemma \ref{lm: fractional Morrey ineq}) that $\alpha$ is continuous in time and $C^{1, \gamma}$  in space with $\gamma = \delta -\frac d p $. Moreover one can show that  $\alpha \in C^{0,1}([0,T]\times \mathbb R^d)$  by similar computations as \cite[Lemma 21]{flandoli_et.al14},  thus we can apply  \cite[Corollary 3.13]{gozzi_russo06} and get
	\begin{align*}
	\left[ 	Y, W \right]_s &= \left[  \alpha(\cdot, X^{t,x} ), W \right]_s 
		= \int_0^s \nabla \alpha (r, X_r^{t,x}) \mathrm dr.
	\end{align*}
	On the other hand, the  covariation calculated using the BSDE \eqref{eq: Y for mild solution} gives 
	\begin{align*}
		&\left[  Y, W \right]_s \\
=& \left[  \Phi ( {X}_T^{t,x})-\int^T_\cdot {Z}_r^{t,x}\mathrm d  {W}_r+\int^T_\cdot f(r, {X}_r^{t,x}, \alpha(r, X_r^{t,x}),\beta(r, X_r^{t,x}))\mathrm d r , W \right]_s \\
		&+ \left[  -w(\cdot, X^{t,x})-\int^T_\cdot\nabla w(r, {X}_r^{t,x})\mathrm d  {W}_r , W \right]_s\\
		=& -  \left[  \int_\cdot^T  Z_r^{t,x} \mathrm d W_r, W \right]_s - \left[  w(\cdot, X^{t,x}), W \right]_s\\
		&- \left[  \int^T_\cdot\nabla w(r, {X}_r^{t,x})\mathrm d  {W}_r  , W \right]_s\\
		=& \int_0^s Z_r^{t,x} \mathrm dr - \int_0^s \nabla w(r, {X}_r^{t,x})\mathrm d  r + \int_0^s \nabla w(r, {X}_r^{t,x})\mathrm d  r \\
		=& \int_0^s  \beta(r, X_r^{t,x} ) \mathrm dr .
	\end{align*}
Therefore $\beta(s, X_s^{t,x}) =  \nabla \alpha(s,X_s^{t,x} )$ for all $s$. Equation \eqref{eq: mild form for alpha and beta} becomes
	\begin{align*}
		\alpha(t)=& P_p({T-t})\Phi + \int^T_tP_p({r-t})\left({ \nabla u(r)\,b(r)}\right ) \mathrm d r \\
		&+\int^T_t P_p({r-t})f(r,\alpha(r),\nabla\alpha(r))\mathrm d r.
	\end{align*}
	We remark that this is exactly the mild formulation of 
	\begin{align}\label{eq: mild form for alpha}
		\left\{\begin{array}{ll}
			&\alpha_t (t,x)  +\frac12 \Delta \alpha (t,x) + {\nabla u(t,x)b (t,x)} + f(t,\alpha(t,x), \nabla \alpha(t,x))=0,\\
			& \alpha(T,x) = \Phi(x),\\
			& \forall (t,x)\in[0,T]\times \mathbb R^d,
		\end{array} \right.
	\end{align}
	where $u$ is the mild solution of \eqref{eq: PDE backward Rd}.  With a very similar proof of Theorem \ref{thm: PDE backward} one can show that there exists a unique mild solution $\alpha \in C([0,T]; H^{1+\delta}_p)$ to \eqref{eq: mild form for alpha}. But by Theorem \ref{thm: PDE backward} we also know that $u$ is  a solution of  \eqref{eq: mild form for alpha} hence we have $\alpha= u$. The claims ${Y}_t^{t,x}= u(t,x) $ and  $ {Z}_t^{t,x}= \nabla u(t,x) $ are thus proved.
\end{proof}

\section{Solution of   FBSDE \eqref{eq: forward-backward SDE}}\label{sc: forward-backward SDE}
\subsection{Some heuristic comments}\label{ssc: heuristc comments}
In this last section we study the forward-backward system (\ref{eq: forward-backward SDE}) recalled again below for ease of reading:
 \begin{equation} \label{eq: FBSDE}
	\left\{
\begin{array}{l}	
 X_s^{t,x} =\  x +  \int_t^s b(r,X_r^{t,x}) \mathrm dr  +\int_t^s \mathrm dW_r,\\ 
	 Y_s^{t,x} = \Phi( X_T^{t,x})  - \int_s^T Z_r^{t,x} \mathrm dW_r + \int_s^T f(r,X_r^{t,x}, Y_r^{t,x},Z_r^{t,x}) \mathrm dr,\\
	 \forall   s\in[t,T].
	 \end{array}\right.
\end{equation}
We will go into more technical  details in Section \ref{subsc: forward SDE} and below, but first we want to make some heuristic comments on the link between the system above and the other FBSDE, given by \eqref{eq: forward-backward SDE transformed}.\vspace{8pt}\\  
If we were in the classical (and smooth enough) case where $b$ is a suitable function, we would be able to change measure  in \eqref{eq: FBSDE}    and apply  Girsanov's theorem: We could  find a new measure $\widetilde{\PP}$ defined by $\mathrm d \widetilde \PP := M_T \mathrm d \PP$ under which   $\widetilde{W}_s := W_s + \int_0^s b(r, X_r^{t,x}) \mathrm d r$ is a Brownian motion. Here $M_s :=  \exp (-\int_0^s b(r, X_r) \mathrm d W_r  - \frac12 \int_0^s b^2(r, X_r) \mathrm d r  )$ is a martingale. Under the new measure $ \widetilde{ {\PP}}$,  the system \eqref{eq: FBSDE} would read
\begin{equation}\label{eq: FBSDE2}
\left\{
\begin{array}{l}
\widetilde{X}_s^{t,x} = x  + \widetilde{W}_s- \widetilde{W}_t,\\
\widetilde{Y}_s^{t,x} = \Phi( \widetilde{X}_T^{t,x})  - \int_s^T  \widetilde{Z}_r^{t,x} \mathrm d \widetilde{W}_r + \int_s^T f(r, \widetilde{X}_r^{t,x},  \widetilde{Y}_r^{t,x},\widetilde Z_r^{t,x}) \mathrm dr\\
\qquad +\int^T_s{ \widetilde{Z}_r^{t,x}b(r, \widetilde{X}_r^{t,x}) }\mathrm{d} r,\\
{\forall }  s\in[t,T], 
\end{array}\right.
\end{equation}
which is exactly equation \eqref{eq: forward-backward SDE transformed} mentioned above.  
In both cases the associated PDE would be the same, namely \eqref{eq: PDE backward intro}, recalled below
\begin{equation} \label{eq: PDE3}
\left\{\begin{array}{l}
 u_t(t,x) +  L^b u(t,x) + f(t,x,u(t,x), \nabla u(t,x))=0, \\
 u(T,x)=\Phi(x),\\
 \forall  (t,x)\in[0,T]\times \mathbb R^d.
\end{array}
\right.
\end{equation}
This can be easily checked by applying It\^o's formula to $u(s, X_s^{t,x})$  (respectively $u(s,\widetilde  X_s^{t,x})$), and identifying $Y$ and $Z$ (respectively $\widetilde Y$ and $\widetilde Z$) with $u$ and $\nabla u$ calculated in $X$ (respectively $\widetilde X$). 

The fact that the same PDE leads to two different FBSDEs can be interpreted analytically by looking at the PDE from two different viewpoints. On one hand we can look at the PDE and the semigroup generated by the Laplacian ($\frac12 \Delta$), which is also the generator of the forward component. In this case the process generated  is a Brownian motion (which is  $X$), so one gets to \eqref{eq: FBSDE2}. Alternatively, we can look at the semigroup generated by the Laplacian \emph{and} the term involving $b$ (that is $L^b=\frac12 \Delta+ (\nabla\cdot) \, b$), which is again the generator of the forward component, but in this case this process is a Brownian motion with drift,   more specifically it is the solution of  $\widetilde X_s = x+\int_t^s b(r,\widetilde X_r) \mathrm dr + \int_t^s \mathrm d \widetilde W_r$. This second viewpoint leads to \eqref{eq: FBSDE}.

Clearly when the drift $b$ is a distribution, this argument is no longer rigorous: We are not able to justify the change of measure (which would involve two measures which are not equivalent). 
 From the analytical point of view, it is unclear to us how to characterize the ``semigroup'' generated by $L^b$. We do not have answers to those questions yet.  
\vspace{7pt}

 What we achieve here instead, is an independent study of the system \eqref{eq: FBSDE}. We will define what a solution is, show its existence (but not uniqueness) and prove rigorously the link between the system \eqref{eq: FBSDE} and the   PDE \eqref{eq: PDE3}.

\subsection{The forward component $\bm X$}\label{subsc: forward SDE}
It is easy to see that the  forward-backward system  \eqref{eq: FBSDE} can be decoupled and the forward component solved first.  
We  define  a solution of \eqref{eq: FBSDE} using  both classical literature about weak solutions of  FBSDEs (see for example \cite{Buckdahn_et_al04, delarue-guatteri06, ma_et.al08}) and the notion of virtual solution for an SDE with distributional drift from \cite{flandoli_et.al14}.
 Here the authors introduced and studied (in the special case where $t=0$) 
equations in $\R^d$ of the form 
\begin{equation}\label{eq: forward SDE}
X_s^{t,x} =\  x +  \int_t^s b(r,X_r^{t,x}) \mathrm dr  +\int_t^s \mathrm dW_r, \qquad s\in[t,T]
\end{equation}
 with drift $b$ being a distribution as specified in the standing assumption, with the extra $L^q$-condition that $b\in L^{\infty}([0,T]; H_q^{-\beta}\cap H_{\tilde q}^{-\beta})$, where $q$ is as usual and $\tilde q : = \frac{d}{1-\beta}$.  In this Section we recall some of their results   for the reader's convenience. Notice that Lemma \ref{lm: properties of psi_n and V^n} is a new result.

 To define a virtual solution we  need to consider the following  auxiliary PDE
\begin{equation}\label{eq: PDE for xi}
\left\{ \begin{array}{l}
	\xi_s(s,y) + L^b \xi(s,y) -(\lambda+1) \xi(s,y)=-b(s,y),  \\
	 \xi(T,y)=0, \\
	  \forall (s,y)\in[0, T ]\times\mathbb R^d .
\end{array} \right.
\end{equation}
  This PDE is similar to \eqref{eq: PDE backward Rd} and can be treated with similar techniques. In \cite[Theorem 14]{flandoli_et.al14} the authors show that the  PDE \eqref{eq: PDE for xi} admits a unique mild solution in $C([0,T], H^{1+\delta}_p)$. This solution enjoys several smoothness properties and in particular it has a continuous version that can be evaluated pointwise and that will be used in the definition of virtual solution and in the construction of the auxiliary SDE below.
  
  By \emph{standard set-up} we mean a quintuple $(\Omega, \mathcal F, P, \mathbb F,(W_t)_t)$ where $(\Omega, \mathcal F,  P)$ is a complete probability space, $\mathbb F$ is a filtration satisfying the usual hypotheses and $W=(W_t)_t$ is an $\mathbb F$-Brownian motion.
According to \cite{flandoli_et.al14} we give the following definition. 
\begin{definition} \label{def: virtual solution} 
\cite[Definition 25]{flandoli_et.al14}
 A  standard set-up $(\Omega, \mathcal F,  P,\mathbb F, (W_t)_t)$ and a continuous stochastic process $X:=(X^{t,x}_s)_s$  on it are said to be  a {\em virtual solution} of \eqref{eq: forward SDE} if $X$ is  $\mathbb F$-adapted and the integral equation
	
	\begin{align}
		\nonumber
		X^{t,x}_{s}=\ &x+\xi(t,x)- \xi(s,X^{t,x}_{s})+(\lambda+1)\int_{t}^{s}\xi(r,X^{t,x}_{r})\mathrm{d}%
		r\\ 
		&+\int_{t}^{s} { ( \nabla \xi(r,X^{t,x}_{r}) + \mathrm I_d)} \mathrm{d}W_{r},\label{eq: virtual solution}
	\end{align}
	
	holds for all $s\in[t,T]$,  $\PP$-a.s.
\end{definition}

To  construct  a virtual solution to \eqref{eq: forward SDE} we transform \eqref{eq: virtual solution}  using the  auxiliary PDE \eqref{eq: PDE for xi} and  we get an auxiliary  SDE  (see equation \eqref{eq: SDE for V} below) which we solve in the weak sense.
 Let us define $\varphi (s,y) :=y+\xi(s,y)$ and let  
\begin{equation}\label{eq: inverse of phi}
\psi(s, \cdot) := \varphi^{-1}(s, \cdot)
\end{equation}
 be the inverse of $y\mapsto \varphi (s,y)  $ for any fixed $s$, which is shown to exist and to be jointly continuous, see \cite[Lemma 22]{flandoli_et.al14}. Let $V$ be the weak solution of the following auxiliary SDE
\begin{align}\label{eq: SDE for V}
	V_s^{t,x} =&\ v + (\lambda+1) \int_t^s \xi(r,\psi(r, V_r^{t,x})) \mathrm dr \nonumber \\
	&+\int_t^s (\nabla \xi(r,\psi(r, V_r^{t,x})) + \mathrm{I}_d) \mathrm dW_r, 
\end{align}
for $s\in [t,T]$, where $\mathrm{I}_d$ is the $d\times d$ identity matrix and  $\xi$ is the solution of \eqref{eq: PDE for xi}. Equation \eqref{eq: SDE for V} is  exactly    \cite[equation (34)]{flandoli_et.al14}, where the authors show that a unique weak solution exists.  Then in  \cite[Theorem 28]{flandoli_et.al14} the authors show existence and uniqueness of a virtual solution according to Definition \ref{def: virtual solution} by making use of the weak solution of the SDE \eqref{eq: SDE for V} with initial condition   $v=\varphi(t,x)= x+\xi(t,x)$. This result is recalled in what follows.
\begin{proposition}\label{pr: existence of a virtual solution}
\cite[Theorem 28]{flandoli_et.al14}
Let	Assumption \ref{ass: f in Rd} hold and let $b \in L^\infty([0,T], H^{-\beta}_{ q}\cap H^{-\beta}_{\tilde q} )$ where $\tilde q : = \frac{d}{1-\beta}$.
Then for every $x\in\mathbb R^d$ and $0\leq t<T$, there exists a unique virtual solution of \eqref{eq: forward SDE} which has the form $X_s^{t,x}=\psi(s, V_s^{t,x})$, where $V$ is the unique weak solution of \eqref{eq: SDE for V} and $\psi$ is given by \eqref{eq: inverse of phi}.
\end{proposition}
Finally let us remark that, although  the transformation $\psi$ appearing in \eqref{eq: SDE for V}  involves a parameter $\lambda$ not included in the original SDE for $X$, the the virtual solution  does not  actually  depend
on $\lambda$. This is a consequence of \cite[Proposition 29]{flandoli_et.al14}.
\vspace{8pt}

The next results are important in the proof of Theorem \ref{thm: Feynman-Kac formula}  below, when we approximate the coefficient $b$ with a smooth sequence $b^n$.
Let us denote by $\psi_n, \varphi_n,\xi^n$ and $V^n$ the same objects as above  associated to equations \eqref{eq: SDE for V} and \eqref{eq: PDE for xi} but with $b$ replaced by a smooth sequence $b^n$. 
In this case it was shown in \cite[Lemma 23 and Lemma 24, (iii)]{flandoli_et.al14}  that the following property holds:

\begin{lemma}\label{lm: convergence of psi_n as b^n}
\cite[Lemma 23 and Lemma 24, (iii)]{flandoli_et.al14} \\
If $b^n\to b$  in $L^\infty([0,T], H^{-\beta}_{q}\cap H^{-\beta}_{\tilde q} )$, then $\xi^n\to \xi$  in $C([0,T], H^{1+\delta}_{p} )$. Moreover, $  \xi^n \to   \xi$ and $\nabla \xi^n \to \nabla \xi$ uniformly in $[0,T]\times \mathbb R^d$.
\end{lemma}

For $\psi_n$ and $V^n$, we have the following result. 
\begin{lemma}\label{lm: properties of psi_n and V^n}
Let  $b^n\to b$  in $L^\infty([0,T], H^{-\beta}_{ q}\cap H^{-\beta}_{\tilde q} )$. Then
	\begin{itemize} 
		\item   [(i)] 
		the functions $\psi_n$ and $\psi$ are jointly $\gamma$-H\"older continuous (for any $\gamma<1-\delta-\beta$) in the first variable and Lipschitz continuous in the second variable, uniformly in $n$, in particular there exists a constant $C>0$ independent of $n$ such that 
		\begin{equation}\label{eq: psi_n is holder-lipshitz continuous}
		|\psi_n(t,x) - \psi_n(s,y)|\leq C (|t-s|^\gamma  + |x-y|).
		\end{equation}
		\item  [(ii)] 
		the moments of $V^n$ can be controlled uniformly in $n$, in particular there exists a constant $C=C(p)>0$ independent of $n$ such that, for every $a>2$, 
		\begin{equation}\label{eq: bounds for p-moments of V^n}
		\EE\lt[|V^n_t - V^n_s|^{a}\rt] \leq C (|t-s|^{a}  + |t-s|^{a/2}).
		\end{equation}		
	\end{itemize}     
\end{lemma}

\begin{proof}
	(i) Let $t,s>0$ and $x,y\in\mathbb R^d$. Then $$|\psi_n(t,x)-\psi_n(s,y)|\le|\psi_n(t,x)-\psi_n(t,y)| +|\psi_n(t,y)-\psi_n(s,y)|.$$
	The first term on the right hand side is bounded by $2|x-y|$ since $\sup_{(t,x)}|\nabla \psi_n(t,x)| <2$ by \cite[Lemma 24 (ii)]{flandoli_et.al14}. 
	The second term can be bounded with a similar  proof as \cite[Lemma 22, Step 3]{flandoli_et.al14} and one gets 
	\[
	|\psi_n(t,y)-\psi_n(s,y)| \leq \frac12 |\psi_n(t,y)-\psi_n(s,y)| + |\xi^n(t,y) - \xi^n(s,y)|.
	\]
	Using the fractional Morrey inequality (Lemma \ref{lm: fractional Morrey ineq})  we have
	\[
	|\xi^n(t,y)-\xi^n(s,y)| \leq C \| \xi^n(t,\cdot)-\xi^n(s,\cdot)\|_{H^{1+\delta}_p}\leq C \| \xi^n\|_{C^{0,\gamma}}|t-s|^\gamma,
	\]
	where $\|\xi^n\|_{C^{0,\gamma}}\le C$ with $C$ independent of $n$ (proof similar to Lemma \ref{lm: properties of solution u and u^n of PDE}, {(i)}).

	(ii) This bound is proven by similar arguments as in Step 3 in the proof of \cite[Proposition 29]{flandoli_et.al14}, with the only difference that the exponent 4 is replaced by $a$ for any $a>2$.
\end{proof}

\subsection{Definition of solution for FBSDE and existence} \label{ssc: virtual-weak solution: def and existence}
  Let us consider  the virtual solution  to the forward equation in \eqref{eq: FBSDE}, which is a \emph{standard  set-up} $(\Omega, \mathcal F, P, \mathbb F,(W_t)_t)$ and a process $(X^{t,x}_s)_s$ that solves \eqref{eq: virtual solution}. We introduce the following definition.

\begin{definition}\label{def: virtual-weak solution FBSDE}
	A {\em virtual-weak solution} to the FBSDE \eqref{eq: FBSDE}  is a standard set-up $(\Omega, \mathcal F, \PP, \mathbb F, (W_t)_t)$ and a triplet of processes $(X^{t,x}, Y^{t,x}, Z^{t,x})$ such that
	\begin{itemize}  
		\item  $X^{t,x}, Y^{t,x}$ and $Z^{t,x}$ are $\mathbb F$-adapted, $X^{t,x} $ and $Y^{t,x}$ are continuous;
		\item  $\PP\left(|\Phi(X_T^{t,x})|+\int_0^T\lt(  |f(s, X_s^{t,x}, Y_s^{t,x}, Z_s^{t,x})| + |Z_s^{t,x}|^2\rt) ds <\infty \right) =1$; 
		\item  $(X^{t,x},Y^{t,x},Z^{t,x})$ verifies,  $\PP$-a.s.,
\begin{equation}\label{eq: FBSDE-virtual} 
	\left\{
\begin{array}{l}	
X^{t,x}_{s}=\  x+ \xi(t,x)- \xi(s,X^{t,x}_{s})+(\lambda+1)\int_{t}^{s}\xi(r,X^{t,x}_{r})\mathrm{d} r\\ 
		\qquad +\int_{t}^{s} { ( \nabla \xi(r,X^{t,x}_{r}) + \mathrm I_d)} \mathrm{d}W_{r},\\ 
	 Y_s^{t,x} = \Phi( X_T^{t,x})  - \int_s^T Z_r^{t,x} \mathrm dW_r + \int_s^T f(r,X_r^{t,x}, Y_r^{t,x},Z_r^{t,x}) \mathrm dr,\\
	 \forall   s\in[t,T].
	 \end{array}\right.
\end{equation}
	\end{itemize}
\end{definition}

{ 
As we can see, the system \eqref{eq: FBSDE-virtual} is decoupled and the backward equation does not involve the rough term $b$, hence using the results of \cite{flandoli_et.al14} we first solve the forward SDE and then  we can apply standard arguments on the BSDE to obtain existence and uniqueness of a (strong) solution $(Y,Z)$ for the BSDE. Of course, when this is put together with the virtual solution $X$ one obtains a virtual-weak solution $(X,Y,Z)$, as demonstrated below. 

\begin{theorem}\label{thm: exist! virtual-weak sol}
Let	Assumption \ref{ass: f in Rd} hold and let $b \in L^\infty([0,T], H^{-\beta}_{ q}\cap H^{-\beta}_{\tilde q} )$.
	Then there exists a unique virtual-weak solution  to the FBSDE system \eqref{eq: FBSDE} given by the standard set-up $(\Omega, \mathcal F, \PP, \mathbb F, (W_t)_t)$ and the triplet $(X^{t,x},Y^{t,x},  Z^{t,x})$, where  the process $X^{t,x}$ and the standard set-up are the unique in law virtual solution of \eqref{eq: forward SDE}, and the couple $ ( Y^{t,x},  Z^{t,x})$ is the unique strong solution of the BSDE in \eqref{eq: FBSDE} for a given forward process  $X$.
\end{theorem}

\begin{proof}
In this proof we will drop the superscript ${t,x}$ for shortness.
\\
By Proposition \ref{pr: existence of a virtual solution}, there exists a unique virtual solution to the forward component in \eqref{eq: FBSDE}, which we denote by $X$ with standard set-up  $(\Omega, \mathcal F, \PP, \mathbb F, (W_t)_t)$. Moreover we know that $X_s = \psi(s, V_s) $, where $V$ is the unique weak solution to the SDE \eqref{eq: SDE for V} and $\psi$ is jointly continuous. 

Standard results on BSDEs (see \cite[Theorem 4.3.1]{zhang17}) can be applied to 
\[
Y_s = \xi + \int_s^T g(r, Y_r, Z_r) \mathrm dr - \int_s^T Z_r \mathrm dW_r 
\]
where $\xi:= \Phi(X_T^{t,x})$ and $g(r,y,z):= f(r, \psi(s, V_s), y,z)$ is a random function. Indeed $\xi$ and  $g$ satisfy \cite[Assumption 4.0.1]{zhang17} because  (i) the filtration we use is the Brownian filtration; (ii) $g$ is $\mathbb F$-measurable in all variables; (iii) $g$ is uniformly Lipschitz in $(y,z)$ with constant $L $ (by Assumption \ref{ass: f in Rd} on $f$); (iv) $\mathbb E[|\xi|^2]<\infty$ because $\Phi$ is bounded and continuous (see Remark \ref{rm: link between Assumptions}) and $\mathbb E[| \int_0^T g(r, 0,0) \mathrm |^2 ]\leq T^2 C^2<\infty$ because $f(r,x, 0,0)$ is uniformly bounded by $C$ according to  Assumption  \ref{ass: f in Rd}.
Thus there exists a unique strong solution $(Y,Z)$ to the BSDE in \eqref{eq: FBSDE} when $X$ is given by the (unique) virtual solution of the forward SDE in \eqref{eq: FBSDE}, which implies that $(X,Y,Z)$ with the standard set-up  $(\Omega, \mathcal F, \PP, \mathbb F, (W_t)_t)$ is the unique virtual-weak solution of \eqref{eq: FBSDE} (because it satisfies  all three bullet points in Definition \ref{def: virtual-weak solution FBSDE}).
\end{proof}

\begin{remark}
Since the BSDE in \eqref{eq: FBSDE} is solved by standard arguments and the forward SDE does not involve $\Phi$, we do not actually need the assumption on $\Phi$ stated in Assumption \ref{ass: f in sobolev space}. Instead it is enough that $\Phi$ is e.g.\ bounded and continuous. 
\end{remark}

Finally we conclude the paper with a Feynman-Kac representation for the virtual-weak solution solution $(X,Y,Z)$ of \eqref{eq: FBSDE} in terms of the solution $u$ to the PDE \eqref{eq: PDE backward Rd}.

\begin{theorem}\label{thm: Feynman-Kac formula}
Let	Assumption \ref{ass: f in Rd} hold and let $b \in L^\infty([0,T], H^{-\beta}_{ q}\cap H^{-\beta}_{\tilde q} )$.
	Then the $(Y,Z)$-component of the unique virtual-weak solution  to the FBSDE system \eqref{eq: FBSDE} is given by  $( u(\cdot,X^{t,x}), \nabla u(\cdot,X^{t,x}))$, where $X^{t,x}$ is the unique virtual solution of \eqref{eq: forward SDE} and $u$ is the solution of PDE \eqref{eq: PDE backward Rd}.
\end{theorem}

}
\begin{proof}
	In this proof we will drop the superscript ${t,x}$ for shortness.
	
	By Remark \ref{rm: link between Assumptions} and Theorem \ref{thm: PDE backward} there exists a unique mild solution to \eqref{eq: PDE backward Rd}, which we denote by $u$. 
{ 
To prove that $(Y,Z) = (  u(\cdot,X), \nabla u(\cdot,X))$	it is enough to show that  $(u(\cdot,X), \nabla u(\cdot,X))$ solves the backward component in \eqref{eq: FBSDE} $\PP$-a.s, with $X$ being the virtual solution of the forward component. Indeed the integrability conditions on $f$ stated in Definition \ref{def: virtual-weak solution FBSDE} are fulfilled because $f$ is Lipschitz continuous in $(y,z)$, bounded at $(t,x,0,0)$ uniformly in $(t,x)$ and $u $ and $\nabla u$ are uniformly bounded by Lemma  \ref{lm: properties of solution u and u^n of PDE};  and  $Z = \nabla u(\cdot,X)$ is square integrable because $\nabla u$ is uniformly bounded again by Lemma \ref{lm: properties of solution u and u^n of PDE}.
}
		
	Let us denote by   $(X^n,Y^n, Z^n)$  the classical strong solution of the FBSDE 
	\begin{equation}\label{eq: forward SDE n}
			\left\{
			\begin{array}{l}
		X_s^{n} = x +  \int_t^s b^n(r,X_r^{n}) \mathrm dr  +\int_t^s \mathrm dW_r,\\
		 Y_s^{n} = \Phi( X_T^{n})  - \int_s^T Z^n_r \mathrm dW_r + \int_s^T f(r,X^n_r, Y_r^{n},Z_r^n) \mathrm dr
		\end{array}\right.
	\end{equation} 
	in $(\Omega, \mathcal F, \PP, \mathbb F, (W_t)_t)$, where  {$b^n\in C([0,T];C_b^1(\R^d;\R^d))$} such that $b^n\to b$ in $L^\infty\left ([0,T]; H^{-\beta}_{q}\cap  H^{-\beta}_{\tilde q}  \right) $.  This strong solution $X^n$ converges in law to $X$ thanks to \cite[Proposition 29]{flandoli_et.al14}.
 Moreover we define $$M^n_s : = \int_t^s Z^n_r \mathrm d W_r \text{ and } F^n_s := \int_t^s f(r,X^n_r,Y^n_r,Z_r^{n} ) \mathrm d r$$ for any $t\leq s\leq T$. Note that from classical theory of BSDEs (see for example \cite{ElKaroui-et.al.97}) we have that $Y^n_s= u^n(s, X^n_s)$ and $Z^n_s = \nabla u^n(s, X^n_s)$.

  We will show that there exists a subsequence of $(X^n, Y^n, Z^n, M^n, F^n, W)$ that converges in law to a limit vector and then we will identify this limit with the components of the solution of  \eqref{eq: FBSDE}.
		
	We   prove the tightness of the sequence 
		$$
		\nu^n=(X^n, {Y}^n, {Z}^n, {M}^n ,{F}^n ,W) 
		$$
	in the space of continuous paths $C([0,T]; \mathbb R^{d^\prime})$, where {$d^\prime=2d+3m+m\times d$}. To do so, we use the following tightness criterion (see for example,  \cite[Corollary 16.9]{kallenberg01}): A sequence of stochastic processes $(\nu^n)_n$ with values in $\mathbb R^d$ is tight in  $C([0,T]; \mathbb R^d)$ if $(\nu_0^n)_n$ is tight and there exists $a, b, C>0$ (independent of $n$) such that 
	\[
	\EE[|\nu^n_r-\nu^n_s|^a]\leq C |r-s|^{1+b}.
	\]
	First note that the initial condition $\nu_0^n$ is deterministic and it converges  pointwise to $\nu_0$, hence it is tight. As for the other bound, we look for an estimate of the quantity
	\begin{align*}
		\EE|\nu^n_r-\nu^n_s|^{a} \leq  & C\EE(|X^n_r-X^n_s|^{a}+ |Y^n_r-Y^n_s|^{a} + |Z^n_r-Z^n_s|^{a}  \\
		& +|M^n_r-M^n_s|^{a}+|F^n_r-F^n_s|^{a}+ |W_r-W_s|^{a}),
	\end{align*}
	for $a>2$, where the constant $C$ depends only on $a$.

The first term is defined as 	$X^n_r=\psi_n(r, V^n_r)$. By Lemma \ref{lm: properties of psi_n and V^n} part (i) we get
	\begin{align*}
		|X^n_r-X^n_s|^{a} & = |\psi_n(r, V^n_r)-\psi_n(s, V^n_s)|^{a} \\
		& \leq C ( |V_r^n-V_s^n| +|r-s|^\gamma )^{a} \\
		& \leq C  (|V_r^n-V_s^n|^{a} +|r-s|^{a\gamma} ),
	\end{align*} 
	and by using Lemma \ref{lm: properties of psi_n and V^n} part (ii) we get
  \begin{align*} 
		\EE|X^n_r-X^n_s|^{a}  &\leq C (\EE|V_r^n-V_s^n|^{a} +|r-s|^{a\gamma} )\\   
		&\leq C  (|r-s|^{a} +|r-s|^{a/2} +|r-s|^{a\gamma} )\\ 
		&\leq C (  |r-s|^{a/2} + |r-s|^{a\gamma}).
	\end{align*}

{
	Next we look at $\EE|Y^n_r-Y^n_s|^{a}$, and using equation \eqref{eq: holder bound for nu} from Lemma \ref{lm: properties of solution u and u^n of PDE} we have 
	\begin{align*}
		\EE|Y^n_r-Y^n_s|^{a} & = \EE|u^n(r, X_r)-u^n(s, X_s)|^{a} \\
		& \leq C\EE( |X_r-X_s| +|r-s|^\gamma )^{a} \\
		& \leq C (\EE|X_r-X_s|^{a} +|r-s|^{a\gamma})\\
		&\leq   C ( |r-s|^{a/2}+  |r-s|^{a\gamma}).
	\end{align*}
	
	The third  term $\EE|Z^n_r-Z^n_s|^{a}$ is done similarly using equation \eqref{eq: holder bound for nu_x} from Lemma \ref{lm: properties of solution u and u^n of PDE} to get $\EE|Z^n_r-Z^n_s|^{a}  \leq   C ( |r-s|^{a\alpha/2}+  |r-s|^{a\gamma})$. 
	
Concerning the term involving $M^n$, using equation \eqref{eq: sup bound for nu_x} from Lemma \ref{lm: properties of solution u and u^n of PDE}, we get
\begin{align*}
\EE|M^n_r-M^n_s|^{a}   &   \leq C \EE\left( \left| \int_t^s \nabla u^n(v, X^n_v) \mathrm dW_v \right |^2 \right)^{a/2}\\
 & \leq C \EE\left(\int_s^r  \left|\nabla u^n(v, X_v)\right |^2 \mathrm dv \right)^{a/2}\\
		&\leq C |r-s|^{a/2}.
\end{align*}

	The last non-trivial term   is 
	\begin{align*}
		\EE|F^n_r-F^n_s|^{a} & = \EE\left(  \int _s^r | f(v, X_v, u^n(v, X_v),  \nabla u^n(v, X_v) ) | \mathrm d v \right)^{a}.
	\end{align*}
The function $f$ inside the integral can be bounded using Assumption \ref{ass: f in Rd} as follows
	\begin{align*}		
&  \sup_{(v,x)} | f(v, x , u^n(v,  x),  \nabla u^n(v, x) ) | \\
\leq & \sup_{(v,x)} | f(v, x , u^n(v,  x),  \nabla u^n(v, x) ) -f(v,x,0,0)| + \sup_{v,x} |f(v,x,0,0)| \\
\leq& \sup_{(v,x)} C (1+ |u^n(v,  x)| + |\nabla u^n(v, x)| ) \\ 
\leq & C  ,
	\end{align*}
 where we have used equation \eqref{eq: sup bound for nu} from Lemma \ref{lm: properties of solution u and u^n of PDE}. Thus 
 \[
 	\EE|F^n_r-F^n_s|^{a} \leq  \EE\left(  \int _s^r  C\mathrm d v \right)^{a} \leq C |r-s|^a. 
 \]

Putting everything together we have
	\begin{align*}
		\EE|\nu^n_r-\nu^n_s|^{a} &\leq C (|r-s|^{a/2} +   |r-s|^{a\gamma}+ |r-s|^{a}),
	\end{align*}
	so choosing $a$ big enough such that  $\min\{a/2, a\gamma\}>1$, then by the tightness criteria we have that $\nu^n$ is tight.

Next we want to identify the limit of $(X^n, Y^n,Z^n, M^n, F^n,W)$. Let us denote by $\nu$ one  limit of $\nu^n $ (or of a subsequence) in  $C([0,T]; \mathbb R^{d^\prime})$, which exists by tightness shown as above. Note that the limit might not be unique. 
 By Skorohod theorem there exists another probability space $(\wt \Omega, \wt F, \wt \PP)$ and other random variables $\wt \nu^n$ and $\wt \nu$ on this space with values in $C([0,T]; \mathbb R^{d^\prime})$  such that $\wt \nu^n\to\wt \nu$, $\wt \PP$-a.s.\ and they have the same laws as the original random variables, in particular $\wt \PP\circ (\wt \nu^n)^{-1} =  \PP\circ ( \nu^n)^{-1}$ and $\wt \PP\circ (\wt \nu)^{-1} =  \PP\circ ( \nu)^{-1}$.
		
		Recall that for fixed $n$ (some of) the components of the vector $\nu^n$ satisfy 
		\[
		Y_s^n = Y_t^n + M^n_s - F^n_s, \quad \PP\text{-a.s.},
		\]
		hence 
		\[
		\wt Y_s^n = \wt Y_t^n  + \wt M^n_s - \wt F^n_s, \quad \wt  \PP\text{-a.s.}.
		\]
		Now taking the limit (along a subsequence) as $n\to \infty$ and by the $\wt \PP$-almost sure convergence of $\wt \nu^n$ to $\wt \nu$ we get
		\[
		\wt Y_s = \wt Y_t + \wt M_s - \wt F_s, \quad \wt  \PP\text{-a.s.},
		\]
		and since $\wt \PP\circ (\wt \nu)^{-1} = \PP\circ ( \nu)^{-1}$ we also have that the components of the limit vector $\nu$ satisfy 
		\[
    	Y_s =  Y_t  + M_s -  F_s, \quad   \PP\text{-a.s.}.
		\]

The last step in the proof consists in showing that the limiting components are of the desired form, for example that $M_s = \int_t^s Z_r \mathrm dW_r $ etc.

We start by  showing the convergence in law of $u^n(s, X^n_s)\to u(s, X_s)$. We do so by using the following result from \cite[Section 3, Theorem 3.1]{billingsley68}: Let $(S,\mu)$ be a metric space and let us consider $S$-valued random variables such that $\xi_n\to\xi$ in law and $\mu(\xi_n , \zeta_n)\to 0$ in probability. Then $\zeta_n\to \xi$ in law.
	
	In the present case, on one hand we have that for any bounded and continuous functional {$G: C([0,T]; \mathbb R^m) \to \mathbb R$}, then $G\circ u$ is also bounded and continuous because  $u$ is uniformly continuous by equation \eqref{eq: holder bound for nu} from Lemma \ref{lm: properties of solution u and u^n of PDE}. Hence by weak convergence of $X^n \to X$ we obtain weak convergence of $G(u(\cdot,X^n))\to G(u(\cdot,X))$, that is $u(\cdot, X^n)\to u(\cdot, X)$ in law. On the other hand   $u^n(\cdot, X^n) - u(\cdot, X^n)\to 0 $ in {$C([0,T]; \mathbb R^m) $}, $\PP$-a.s., because $u^n\to u$ uniformly by Lemma \ref{lm: continuity of u wrt approximation} part (ii), hence  $|u^n(\cdot, X^n) - u(\cdot, X^n)|\to 0 $ in probability. These two facts imply the convergence in law of $u^n(s, X^n_s)\to u(s, X_s)$ by  \cite[Section 3, Theorem 3.1]{billingsley68}. A similar argument can be applied to  $\nabla u^n(s, X^n_s)\to  \nabla u(s, X_s)$ by using equation \eqref{eq: holder bound for nu_x} instead of \eqref{eq: holder bound for nu}.

	Similarly as  above, one can see that the convergence in law means that the components $Y$ and $Z$ in the limit vector $\nu$ satisfy $Y_s=u(s, X_s)$   and  $Z_s=\nabla u(s, X_s)$ $\PP$-a.s.\ in {$ C([0,T]; \mathbb R^m) $ and $C([0,T]; \mathbb R^{m\times d})$}, since $Y^n_s=u^n(s, X^n_s)$ and $Z^n_s=\nabla u^n(s, X^n_s)$.
}

{For the component $F$, we use the continuity assumption of $f$ in  $(x,y,z) $ and the continuity of $u$ and $\nabla u$ in $x$ to show that the map $$X_\cdot^n\mapsto \int^\cdot_t f(r,X^n_r, u(r, X^n_r), \nabla u(r, X^n_r))\d r$$ composed with any bounded and  continuous functional {$G: C([0,T]; \mathbb R^m) \to \mathbb R$} is still bounded and continuous, hence we have that  
	\begin{align*}
	&\EE\lt[G\lt(\int^\cdot_t f(r,X^n_r, u(r, X^n_r), \nabla u(r, X^n_r))\d r\rt)\rt] \\
	&\to \EE\lt[G\lt(\int^\cdot_t f(r,X_r, u(r, X_r), \nabla u(r, X_r))\d r\rt)\rt]
	\end{align*}
	  from the weak convergence of $X^n\to X$. 
 Moreover the convergence in probability of $ |u^n(\cdot, X^n) - u(\cdot, X^n)|\to 0  $  in {$C([0,T]; \R^m)$} and  the Lipschitz character of $f$ imply that 
	 \begin{align*}
	 &\int_t^\cdot |f( r,X_r^n, u^n( r, X_r^n), \nabla u^n(r, X^n_r) )-f(r, X_r^n, u(r, X^n_r), \nabla u(r, X^n_r) )| \d r \\
	 &\leq L \int_t^\cdot \left (|  u^n(r, X^n_r)- u(r, X^n_r)|+ | \nabla  u^n(r, X^n_r)-\nabla u(r, X^n_r)|\right) \d r \\
	 &\leq C \int_t^\cdot |  u^n(r, X^n_r)- u(r, X^n_r)| \d r \to 0
	 \end{align*} 
	 in probability  in {$C([0,T]; \R^m)$}.
	Hence applying  again    \cite[Section 3, Theorem 3.1]{billingsley68} we obtain that $	\int^\cdot_t f(r,X^n_r, u^n(r, X^n_r),\nabla u^n(r, X^n_r))\d r $ converges to $\int^\cdot_t f(r,X_r, u(r, X_r), \nabla u(r, X_r))\d r $
	in law.  Thus, for the component $F$ of the limit vector $\nu$ we have that 
	$$
	F_s=\int_t^s f(r, X_r, u(r,X_r), \nabla u (r, X_r) ) \mathrm dr = \int_t^s f(r, X_r, Y_r, Z_r ) \mathrm dr,
	$$
	 $\PP${-a.s.}. 
	 
	It remains to show that $M_s=\int_t^s Z_r \mathrm dW_r$, $\PP$-a.s. This follows from \cite[Theorem 7.10]{kurtz_protter96} (see also \cite[Section 2.2]{engelbert-peskir14}) and from the fact that $Z^n\to Z$ weakly. 
	
	Putting everything together and using the fact that 
	$$
	Y_t= Y_T - \int_t^T Z_r \mathrm dW_r + \int_t^T f(r, X_r, Y_r, Z_r ) \mathrm dr 
	$$ 
	we have
	\begin{align*}
		Y_s &=  Y_t + \int_t^s Z_r \mathrm dW_r   -  \int_t^s f(r, X_r, Y_r, Z_r ) \mathrm dr  \\
		&  = Y_T - \int_s^T Z_r \mathrm dW_r +  \int_s^T f(r, X_r, Y_r, Z_r ) \mathrm dr, \ \ \PP\text{-a.s.},
	\end{align*}
	 where $Y_s=u(s, X_s)$   and  $Z_s=\nabla u(s, X_s)$, as wanted.
	 }
\end{proof}

\noindent {\bf Acknowledgement:} The authors would like to thank Tiziano De Angelis for various suggestions and fruitful discussions on this topic, {and the anonymous referees for reading the manuscript and providing useful hints that led to several improvements.}

\bibliographystyle{plain}
\bibliography{biblio-fbsde}

\end{document}